\documentclass[microtype]{gtpart}     
  \usepackage{amscd,comment}
  \usepackage{geometry}
  \usepackage{graphicx,epstopdf}
  \usepackage{color}
  \usepackage{enumerate}
  \usepackage{xspace}
  \usepackage{tipa}
\usepackage{latexsym}
  \usepackage{epsfig}
  \usepackage{pinlabel}
  \usepackage[all]{xy}

  \usepackage{hyperref}

  \newcommand{\calK}{\mathcal{K}}

  \newcommand{\calR}{\mathcal{R}}

  \newtheorem{theorem}{Theorem}[section]
  \newtheorem{proof of the main theorem}[theorem]{Proof of the Main Theorem}
  \newtheorem{proposition}[theorem]{Proposition}

  \newtheorem{corollary}[theorem]{Corollary}
  
  \newtheorem{lemma}[theorem]{Lemma}
  
  \newtheorem*{conjecture*}{Conjecture}

  \theoremstyle{definition}
  \newtheorem{definition}[theorem]{Definition}

  \newtheorem*{claim*}{Claim}

  \newtheorem{example}[theorem]{Example}
  
  \newtheorem*{question 1 *}{Question 1}
  \newtheorem*{question 2 *}{Question 2}
  \newtheorem*{question 3 *}{Question 3}
  \newtheorem*{answer*}{Answer}
  \newtheorem*{application*}{Application}
  \newtheorem*{ideas*}{ideas}
  \theoremstyle{remark}
  \newtheorem{remark}[theorem]{Remark}
  \newtheorem*{remark*}{Remark}
  \newtheorem*{theorem*}{Theorem}

  \newcommand{\Aut}{\ensuremath{\operatorname{Aut}}\xspace}
  
  \newcommand{\Teich}{{Teichm\"uller }}
  \newcommand{\Ham}{{Hamenst\"adt }}

  \newcommand{\param}{{\mathchoice{\mkern1mu\mbox{\raise2.2pt\hbox{$
  \centerdot$}}
  \mkern1mu}{\mkern1mu\mbox{\raise2.2pt\hbox{$\centerdot$}}\mkern1mu}{
  \mkern1.5mu\centerdot\mkern1.5mu}{\mkern1.5mu\centerdot\mkern1.5mu}}}
\renewcommand{\setminus}{{\smallsetminus}}

   \newcommand{\F}{{\mathbb{F}}}

\newcommand{\FF}{\ensuremath{\mathcal{F}\mathcal{F} } }

\newcommand{\out}{\ensuremath{\mathrm{Out}(\mathbb{F}) } }
\newcommand{\aut}{\ensuremath{\mathrm{Aut}(\mathbb{F}) }}
\renewcommand{\F}{\ensuremath{\mathbb{F} } }

\newtheorem*{hypnas}{\textbf{Theorem}~\ref{hypnas}}

\newtheorem*{suffdiff}{\textbf{Proposition}~\ref{suffdiff}}
\newtheorem*{RH}{\textbf{Theorem}~\ref{RH}}
\newtheorem*{hypffc}{\textbf{Theorem}~\ref{hypffc}}
\newtheorem*{suffdiffhyp}{\textbf{Theorem}~\ref{suffdiffhyp}}

\title[Geometry of free extensions of free groups ]{Geometry of extensions of free groups 
via automorphisms with fixed points on the complex of free factors}

\author{Pritam Ghosh}
\givenname{Pritam}
\surname{Ghosh}
\address{Ashoka University\\
  Haryana 131029, India}
\email{pritam.ghosh@ashoka.edu.in}

\author{Funda G\"ultepe}
\givenname{Funda}
\surname{ G\"ultepe}
\address{Department of Mathematics and Statistics\\
 University of Toledo\\
 Toledo, OHIO}
\email{funda.gultepe@utoledo.edu}

\arxivreference{2404.19058}

\theoremstyle{definition}
\begin{document}

\begin{abstract}
We give conditions of an extension of a free group to be hyperbolic and relatively  hyperbolic using the dynamics of the action of $\out$ on the complex of free factors combined with the weak attraction theory. We work with subgroups of exponentially growing outer automorphisms and instead of using a standard pingpong argument with loxodromics, we allow fixed points for the action and investigate the geometry of the extension group when the fixed points of the automorphisms on the complex of free factors are sufficiently far apart.

  \vspace{0.5cm}

\end{abstract}
\maketitle


\section{Introduction}
\label{intro}

Let $\F$ be a free group of finite rank $\geq 3$. Similar to that of mapping classes for a surface, there is a dichotomy  for elements of  the group  of outer automorphisms $\out$ of the free group $\F$, in terms of their \emph{growth}. An element $\phi\in \out$ is called \emph{exponentially growing} if for some conjugacy class $[w]$ of an element $w\in \mathbb F$,  the word length of $\phi^{i}([w])$ grows exponentially with $i$ for any fixed generating set of $\F$. We will call a subgroup of $\out$ exponentially growing if all of its elements are. For an exponentially growing subgroup $\mathcal{Q}<\out$,  we are interested in understanding the geometry of the extension $E_{Q}$ given by the short exact sequence
\[ 1\to \F \to E_{Q}  \to \mathcal{Q} \to 1\]
that is induced from the sequence 
 \[1\to \F \to \aut \to \out \to 1. \]
 
The extension group $E_{Q}$ is the pullback of $\mathcal{Q}$ to $\aut$, hence  a subgroup of $\aut$. 
 
  When $\mathcal{Q}$ is free group of rank $\geq 2$, which is our main focus for this paper, we say that $ E_{Q}$ is a free-by-free group and it is known that $ E_{Q}\cong \F\rtimes \widehat{\mathcal{Q}}$, where $\widehat{\mathcal{Q}}$ is some (any) lift of $\mathcal{Q}$ to $\aut$. 
In this paper we give necessary and sufficient conditions for $E_Q$ to be  hyperbolic using the dynamical information of the generators of $\mathcal{Q} < \out$  obtained from train track maps of free group automorphisms, as well as via their actions on the complex of free factors $\FF$. 

 An outer automorphism  $\phi$ is \emph{atoroidal} (or hyperbolic) if no power of $\phi$ fixes a nontrivial conjugacy class in $\F$. Atoroidal automorphisms are a special class of exponentially growing elements of $\out$. Using train-track theory, Bestvina-Feighn-Handel \cite{BFH-00} developed dynamical invariants for exponentially growing outer automorphisms of $\F$, called \emph{attracting and repelling laminations}. Namely, associated to each exponentially growing outer automorphism $\phi$ (respectively $\phi^{-1}$), we have finitely many invariant sets of biinfinite \emph{lines} in the (compactified Cayley graph of) Gromov hyperbolic space $\F\cup\partial\F$. These sets of lines are called \emph{attracting laminations} (respectively \emph{repelling laminations}) and denoted by $\mathcal{L}^+(\phi)$ (respectively $\mathcal{L}^-(\phi)$). Atoroidal outer automorphisms are characterized by the property that for every conjugacy class $[w]$, the sequence $\phi^i([w])$ converges to some element of $\mathcal{L}^+(\phi)$ as $i\to\infty$ (Lemma 3.1, \cite{Gh-20}). Ghosh showed that these sets were central to describing the Cannon-Thurston laminations for the extension group $\F\rtimes \langle \phi \rangle$ (Lemma 4.4, \cite{Gh-20}). We use $\mathcal{L}^\pm(\phi)$ to denote the union of the two lamination sets $\mathcal{L}^+(\phi)$ and $\mathcal{L}^-(\phi)$. 

We show that $\mathcal{L}^\pm(\phi)$ is instrumental in characterising hyperbolicity of the extension group $E_Q$, when $Q$ is a free group of rank $\geq 2$.  Our theorem \ref{hypnas} shows that in this case the extension group is hyperbolic, provided we enforce the condition that the attracting and repelling laminations associated to the generators of $Q$ to be pairwise disjoint. 
We also require throughout the paper that our outer automorphisms are \emph{rotationless} (see Section \ref{sec:CT}), a condition that ensures that we will be able  to use special relative train track maps, which can be achieved by taking a uniform power of automorphisms (\cite{FH-11}). 

\begin{theorem}\label{hypnas}
   Let $\phi_1,  \ldots , \phi_k $ be a collection of  exponentially growing outer automorphisms  such that no pair of automorphisms have a common power. 
  Then the following are equivalent:
    \begin{enumerate}
        \item  Each $\phi_i$ is an atoroidal outer automorphism and $\mathcal L^{\pm}(\phi_i)\cap  \mathcal L^{\pm}(\phi_j)=\emptyset$ for all $i\neq j$.
        \item  There exists $M > 0$ such that for all $m_i \geq M$, $Q = \langle \phi_1^{m_1},\cdots,\phi_k^{m_k} \rangle$ is a free group and the extension group $\F\rtimes \widehat{Q}$ is hyperbolic (where $\widehat{Q}$ is any lift of $Q$).
    \end{enumerate}
\end{theorem}
 Theorem \ref{hypnas} is comparable to Theorem 1.3  of Farb and Mosher (\cite{FaM-02}) and the main Theorem of \cite{Ham05} where hyperbolicity of the extension group $E_{Q}=\pi_1(S) \rtimes Q $ is characterized by the \emph{convex cocompactness} of a free subgroup $Q$ of  the mapping class group $MCG(S)$ of a closed surface $S$. 
 
We say that  $\phi_i, \phi_j$, $i\neq j$ are \emph{independent} if generic leaves of  elements of $\mathcal{L}^\pm(\phi_i)$ and  $\mathcal{L}^\pm(\phi_j)$ are not asymptotic to each other. Hence, Theorem \ref{hypnas} can be written as; 

\begin{theorem*}
   Let $\phi_1,  \ldots , \phi_k $ be a collection of  exponentially growing outer automorphisms  such that no pair of automorphisms have a common power. 
  Then, 
    $\phi_1,  \ldots , \phi_k $ is a purely atoroidal (all infinite order elements are atoroidal), pairwise independent 
    collection if and only if 
        the extension group $\F\rtimes \widehat{Q}$ for a free group $Q = \langle \phi_1^{m_1},\cdots,\phi_k^{m_k} \rangle$ generated by high enough $m_i$'s is hyperbolic. 
   
\end{theorem*}

The interest in the geometry of extensions of groups started with Thurston (\cite{Tharxiv}) where the $\mathbb Z$--extension of $\pi_1(S)$ of a closed surface $S$ is hyperbolic if and only if the cyclic group is generated by a pseudo-Anosov element of the mapping class group. 
 Bestvina-Feighn \cite{BF-92} and Brinkmann \cite{Br-00} have shown that when $Q$ is an infinite cyclic subgroup of $\out$, $E_Q$ is hyperbolic if and only if $\phi$ is atoroidal. 

 We remark that we do not need $Q$ to be free to prove that the collection is independent and purely atoroidal.

\subsection{Fixed points on $\FF$ and hyperbolicity of the extension}

Free factor complex $\FF$ of the free group $\F$ is a simplicial complex whose vertices are conjugacy classes of non-trivial free factors of $\F$. Two vertices corresponding to free factors  are connected by an edge if one includes the other as subgroup, up to conjugation. $\FF$ is  Gromov-hyperbolic (\cite{H3,BF-12}) and in many other ways shows resemblance to the curve complex. Hence the action of $\out$ on $\FF$ is reminiscent of the action of the mapping class group on the curve complex. 

An element of $\out$ is \emph{fully irreducible} if and only if it has no periodic orbits on $\FF$. It is well known that fully irreducible elements of $\out$ act loxodromically on $\FF$ (\cite{BF-12}) and this was exploited by Dowdall-Taylor \cite{DowTay1} to produce hyperbolic extensions of free groups. Their theorem can be compared to those of  Kent--Leininger (\cite{KL2008}), and \Ham  (\cite{Ham05}) characterizing convex cocompactness, hence the hyperbolicity of the extension, with quasi isometric embedding of the group in the curve complex . As such the automorphisms of \cite{DowTay1} are fully irreducible.  

We take a different approach in this paper and  produce hyperbolic extensions of groups whose elements are not necessarily fully irreducible, hence we allow automorphisms to have fixed points on $\FF$. The following result shows that the only condition one needs for hyperbolicity is that they do not have a common fixed point on $\FF$.  
\begin{theorem}\label{hypffc}
     Let $\phi_1,  \ldots , \phi_k $ be a collection of atoroidal elements which do not have a common power. If no pair $\phi_i, \phi_j$, $i\neq j$ fixes a common vertex in the free factor complex $\FF$, then
     \begin{enumerate}
         \item There exists $M > 0$ such that for all $m_i \geq M$, $Q=\langle \phi_1^{m_1},\cdots,\phi_k^{m_k} \rangle$ is a free group, 
         \item Extension group $\F\rtimes \widehat{Q}$ is a hyperbolic group (where $\widehat{Q}$ is any lift of $Q$).
     \end{enumerate}
\end{theorem}

Our Theorem \ref{hypffc} hence responds to the question of what happens when a pair $\phi, \psi$ in a collection of automorphisms do have common fixed points for their action on $\FF$. 

\subsection {Sufficiently different automorphisms, non-attracting sink and the geometry of the extension group}
One of the motivations for this work is to establish a connection between the dynamics of the action of $\out$ on the free factor complex and the dynamical data that we have from the train-track theory. We believe that this is a theme that has barely been explored and a lot can be learned if we are able exploit this connection. 

We call a collection of automorphisms  \emph{sufficiently different} if no pair of automorphisms have a common power  and for each pair of outer automorphisms in the collection, the distance between the fixed vertices  in $\FF$ is at least $2$ (if such vertices exist) (see section \ref{sec:suffdiff}). Lemma \ref{pfifbf} gives a simple criterion for constructing free subgroups of $\out$ using sufficiently different elements.

Given a short exact sequence $1 \to \F \to E \to Q \to 1$ of finitely generated groups, we say that $E$ has a \textbf{\emph{cusp-preserving}} relatively hyperbolic structure if there exists a collection of finitely generated subgroups $\{H_i\}$ of $\F$ such that $Q$ preserves conjugacy class of each $H_i$ and $E$ is (strongly) hyperbolic relative to the collection $\{N_E(H_i)\}$ of normalizers of $H_i$ in $E$. Our main aim in Section \ref{sec:cusprelhyp} is to address the question of when a free subgroup $Q$ of $\out$ can yield relatively hyperbolic extension $E$ with the cusp-preserving property. 

In the proposition below we conclude that being sufficiently different is  an obstruction to cusp-preserving relative hyperbolicity of the extension group, except in one particular case.  Equivalently, it is an obstruction to having an \emph{admissible subgroup system} (\cite{GGrelhyp}, Section \ref{sec:K}) for the collection of $\phi_i$'s. Admissible subgroup system is a malnormal collection of subgroups with some special properties,  and the \emph{nonattracting sink} $\mathcal{K}^*_\phi$ of an automorphism $\phi\in\out$, which is developed in \cite{GGrelhyp} as an example of an admissible subgroup system, carries all conjugacy classes which do not grow exponentially under iteration by $\phi$ (see Lemma \ref{sinkprop}). Nonattracting sink can be computed explicitly using relative train-track maps.

For a collection $\phi_1, \cdots, \phi_k$ of pairwise sufficiently different exponentially growing outer automorphisms, let $\mathcal{K}^*_i$ be the nonattracting sink of $\phi_i$.

\begin{suffdiff}\emph{Let $\phi_1, \cdots, \phi_k $ be a collection of  exponentially growing, sufficiently different  outer automorphisms. If all the nonattracting sinks are nonempty then exactly one of following conditions are true:
    \begin{enumerate}
        \item $ \phi_1, \cdots, \phi_k $ is a collection of fully irreducible geometric outer automorphisms induced by some pseudo-Anosov homemorphisms of the same compact surface with one boundary component. 
        Moreover, all sufficiently high powers of $\phi_i$' s generate a free group $Q$ such that $E_Q$ has a cusp-preserving relatively hyperbolic structure.
       \item $\mathcal{K}^*_i \neq \mathcal{K}^*_j$ for some $i\neq j$ and the extension $E_Q$ of a free group $Q$, which is generated by sufficiently high powers of $\phi_i$'s, cannot have a cusp preserving relatively hyperbolic structure. 
    \end{enumerate}}
\end{suffdiff} 

 We have the following theorem characterizing the hyperbolicity of $E_Q$ using sinks:

\begin{theorem}\label{suffdiffhyp}
      Let $\phi_1, \ldots , \phi_k$ be a  collection of exponentially growing and sufficiently different  outer automorphisms.
      Let $Q$ be a free group  generated by sufficiently high powers of $\phi_i$'s.     
      Then, $\F\rtimes \widehat{Q}$ is a hyperbolic group if and only if $\mathcal{K}^*_i = \emptyset$ for all $i\in \{1,\cdots, k \}$.
\end{theorem}
We remark that in this  theorem $Q$ is not necessarily \emph{convex cocompact} subgroup of $\out$ in the sense of \cite{HamHenStab,DowKT}  and \cite{DowTay2}, since our groups do not quasi isometrically embed in $\FF$. Hence our theorem cannot be obtained via their work. In fact, as far as our knowledge goes, this theorem cannot be obtained by studying the action of $Q$ on any known hyperbolic simplicial complex with a nice $\out$ action.

In a similar vein, we obtain the following necessary and sufficient condition in terms of nonattracting sinks, regarding cusp-preserving relative hyperbolicity of the extension group when $Q$ is generated by sufficiently different collection of outer automorphisms.

\begin{RH}\emph{
 Let $\phi_1, \cdots, \phi_k$ be a collection of  pairwise sufficiently different   and exponentially growing outer automorphisms. Let $\mathcal{K}^*_i$ be the nonattracting sink of $\phi_i,\,\,i \in \{1, \cdots, k\}$ and assume that $\mathcal{K}^*_j\neq \emptyset$ for some fixed $j$. Then the following are equivalent:
 \begin{enumerate}
     \item $\F\rtimes \widehat{Q}$ has a cusp preserving relatively hyperbolic structure where $Q$ is a free group generated by sufficiently large exponents of $\phi_i$'s.
     \item $\mathcal{K}^*_i = \mathcal{K}^*_j$ for all $i$. 
 \end{enumerate}}
\end{RH}
To summarize, given a pair of sufficiently different, exponentially growing elements of $\out$: All sinks empty gives hyperbolic extensions (Theorem \ref{suffdiffhyp}). No sink empty gives relatively hyperbolic extensions with cusp-preserving structure only for geometric fully-irreducibles (Proposition \ref{suffdiff}). Some sink nonempty gives relatively hyperbolic extensions with cusp-preserving structure only if all sinks are nonempty and are equal (Theorem \ref{RH}).  

\subsection{Plan of the paper:}

Section \ref{Prelim} collects all the definitions and the tools we use.

In Section \ref{sec:3} we investigate obstructions to (relative) hyperbolicity using the dynamics of sufficiently different automorphisms on $\FF$. In this section we also discuss the notion of a sink and include proof of Proposition \ref{suffdiff}.

In Section \ref{sec:4} we continue working  with a specific type of exponentially growing automorphism; a reducible automorphism that is fully irreducible on a free factor (a \emph{partial} fully irreducible). We investigate the conditions on the small displacement sets on $\FF$ of partial fully irreducible automorphisms that determine the geometry of the extension groups they generate.   

Section \ref{sec:Hyp} includes proofs of Theorems \ref{hypnas}, \ref{hypffc} and \ref{RH}. In this section we discuss further the notion of sink and how the sinks of automorphisms determine the geometry of an extension group generated by those automorphisms and prove Theorem \ref{suffdiffhyp}.  We work with partial and relatively fully irreducible elements  directly to exhibit examples of relatively hyperbolic free group extensions. 

Section \ref{sec:applications} discusses further problems such as characterization of non relative hyperbolicity; in this section we give an example which could help characterize non relative hyperbolicity of a free group extension. 

\subsection {Acknowledgements} 
We thank Spencer Dowdall, Chris Leininger, Lee Mosher,
Pranab Sardar, and Sam Taylor for their remarks and comments on the earlier version of the paper. 
We also would like to thank the anonymous referee for carefully reading the paper and providing numerous suggestions that have improved the exposition.

The first author was supported by the Ashoka University faculty research grant. The second author was partially
supported by NSF grant DMS-2137611.

\section{Preliminaries}\label{Prelim}
\subsection{Marked graphs, circuits and path: }

 A \textit{marked graph} is a graph $G$ which is a core graph (a graph with no valence 1 vertices) that is equipped with  a homotopy equivalence to the rose $m: G\to R_n$ (where $n = \text{rank}(\F)$). Thus the fundamental group of $G$  can be identified with $\F$ (up to inner automorphisms). 
 A \textit{circuit} in a marked graph is a locally injective and continuous map
 of $S^1$ into $G$. The set of circuits in $G$ can be identified
 with the set of conjugacy classes in $\F$.

A \emph{path} is an immersion of the interval $[0, 1]$ into $G$ with endpoints at vertices of $G$. Every path can be written as a finite concatenation of edges of $G$, so that there is no backtracking.

 A \emph{line} $\ell$ is a bi-infinite concatenation $\ell = \ldots E_{i-1} E_i E_{i+1} \ldots$, $i\in \mathbb{Z}$ of edges of $G$ without backtracking \emph{i.e.} so that $E_{j-1}^{-1} \neq E_j\neq E_{j+1}^{-1}$ (where $E_{s}^{-1}$ is $E_s$ traveled in opposite direction).   A \emph{ray} $\gamma$ is ``one-sided'' infinite concatenation of edges $\gamma = E_i E_{i+1} \ldots $, $i\in \mathbb{Z}_{\geq 0}$ in $G$ without backtracking. Two lines are said to be asymptotic if they have a common sub-ray. 
  Any continuous map $f$ from $S^1$ or $[0,1]$ to $G$ can be \textit{tightened} to a circuit or path, in other words is freely homotopic to a locally injective and continuous map. Tightened image of a path $\alpha$ under $f$ will be denoted by $f_{\#}(\alpha)$ and we will  not distinguish between circuits or paths that differ by a homeomorphism of their respective domains.

\subsection{Topological representative, EG strata, NEG strata:}
\label{sec:2}
 A \textit{filtration}  of a marked graph $G$ is a strictly increasing sequence $G_0 \subset G_1 \subset \cdots \subset G_k = G$ of subgraphs $G_r$ with no isolated vertices. The filtration is \textit{$f$-invariant} if $f(G_r) \subset G_r$ for all $r$.

 The \textit{stratum of height $r$} is a subgraph $H_r = G_r \setminus G_{r-1}$ together with the  endpoints of edges.  The minimum $r$ such that a  subset of $G$ is contained in $G_r$ is called the \textit{height} of the subset.

For an $f$--invariant filtration, 
the square matrix $M_r$ whose $j^{\text{th}}$ column records the number of times
the image of an edge $e_j$ under $f$ intersects the other edges in $H_r$ is called 
the \textit{transition matrix}  of the stratum $H_r$. 
	 $M_r$ is said to be irreducible if for each  $i,j$, the $i,j$--th entry of some power of $M_r$ is nonzero. In this case we say that the associated stratum $H_r$ is also \textit{irreducible}. When $H_r$ is irreducible, the Perron-Frobenious theorem states that the matrix $M_r$ has a unique eigenvalue $\lambda \ge 1$, called the \textit{Perron-Frobenius eigenvalue}, for which some associated eigenvector has positive entries.
If $\lambda=1$ then $H_r$ is a \textit{nonexponentially growing (NEG) stratum}  whereas if $\lambda>1$ we say that $H_r$ is an \textit{exponentially growing (EG) stratum}.

An automorphism $\phi\in\out$ can be represented by a  homotopy equivalence $f:G\rightarrow G$ that takes vertices to vertices and edges to edge-paths of a marked graph $G$ with marking $\rho: R_n \rightarrow G$, called a \textit{topological representative}. Topological representative $f$ preserves the marking of $G$, in other words $\overline{\rho}\circ f \circ \rho: R_n \rightarrow R_n$ represents $R_n$. A nontrivial path $\alpha$  in G is a \emph{periodic Nielsen path} if $f^k_{\#}(\alpha)=\alpha$ for some $k$, where the smallest such  $k$ is called the \emph{period}. $\alpha$ is a Nielsen path if $k=1$. A
periodic Nielsen path is \emph{indivisible} (iNP) if it cannot be written as a concatenation of 
nontrivial periodic Nielsen paths.

Given a topological representative $f: G\to G$ let $Tf$ be such that $Tf(E)$ is the first edge in the edge path associated to $f(E)$. A \emph{turn} is a pair of edges $\{E_i, E_j\}$ of edges and we let $Tf(E_i,E_j) = (Tf(E_i),Tf(E_j))$ making $Tf$ a map that takes turns to turns. We say that a
non-degenerate (i.e, $i\neq j$) turn is \emph{illegal} if for some iterate of $Tf$ the turn becomes degenerate; otherwise the
 turn is legal. A path is said to be \emph{legal path} if it contains only legal turns and it is $r-legal$ if it is of height $r$ and all its illegal turns are in $G_{r-1}$.

 \textbf{Relative train track map.} 
 Given $\phi\in \out$ and a topological representative $f:G\rightarrow G$ with a filtration $G_0\subset G_1\subset \cdot\cdot\cdot\subset G_k$ which is preserved by $f$,
 we say that $f$ is a relative train track map if the following conditions are satisfied: \label{rtt}
 \begin{enumerate}
  \item $f$ maps $r$-legal paths to  $r$-legal paths.
  \item If $\gamma$ is a path in $G_{r-1}$ with endpoints in $G_{r-1}\cap H_r$   then $f_\#(\gamma)$ is non-trivial.
  
  \item If $E$ is an edge in $H_r$ then $Tf(E)$ is an edge in $H_r$. In particular,
  every turn consisting of a direction of height $r$ and one of height $< r$  is legal.
 \end{enumerate}
 
By \cite[Theorem 5.1.5]{BFH-00}  every  $\phi\in\out$ has a relative train-track map representative which satisfies some useful conditions in addition to the ones we listed above. 

\subsection{Weak topology}
\label{sec:weak} Given a graph $G$, we define an equivalence relation on the set of all paths, rays, lines, and circuits in $G$ by declaring two of paths/rays/lines/circuits to be equivalent if they differ by a homeomorphism of their domains. 

 Let $\widehat{\mathcal{B}}(G)$ denote the compact space of equivalence classes of circuits and finite paths, rays and lines in a graph $G$, whose endpoints (if any) are vertices of $G$.  
For each finite path $\gamma$ in $G$, we denote by $\widehat{N}(G,\gamma)$ the  set of all paths and circuits in $\widehat{\mathcal{B}}(G)$ which have $\gamma$ as its subpath. The collection of all such sets gives a basis for a topology on $\widehat{\mathcal{B}}(G)$ 
 called \textit{weak topology}. Let $\mathcal{B}(G)\subset \widehat{\mathcal{B}}(G)$ be the compact subspace of all lines in $G$ with the induced topology. 

 Two distinct points in $\partial \mathbb{F}$ determine a \emph{line}, up to reversing the direction. Let $\widetilde{\mathcal{B}}=\{ \partial \mathbb{F} \times \partial \mathbb{F} - \vartriangle \}/(\mathbb{Z}/2 \mathbb Z)$ be the set of pairs of boundary points of $\F$ where $\vartriangle$ is the diagonal and $\mathbb{Z}_2$ acts by interchanging factors. We can give the weak topology to
$\widetilde{\mathcal{B}}$, induced by the Cantor topology on $\partial \mathbb{F}$.

$\mathbb{F}$ acts on $\widetilde{\mathcal{B}}$ and the quotient space $\mathcal{B}=\widetilde{\mathcal{B}}/\mathbb{F}$ is compact but non-Hausdorff as such the topology is called  \emph{weak}. The quotient topology is also called the \textit{weak topology}. For any marked graph $G$, there is a natural identification $\mathcal{B}\approx \mathcal{B}(G)$.

  We call the  elements of $\mathcal{B}$ \textit{lines} as well. A lift of a line $\gamma \in \mathcal{B}$ is an element  $\widetilde{\gamma}\in \widetilde{\mathcal{B}}$ that projects to $\gamma$ under the quotient map and the two elements of $\widetilde{\gamma}$ are called its endpoints.



  A line (or a path) $\gamma$ is said to be \textit{weakly attracted} to a line (path) $\beta$ under the action of $\phi\in\out$, if  for some  $k$  $\phi^k(\gamma)$ converges to $\beta$ in the weak topology, in  other words, if any given finite subpath of $\beta$ is contained in $\phi^k(\gamma)$ for some  $k$. Similarly if we have a homotopy equivalence $f:G\rightarrow G$,  a line(path) $\gamma$ is said to be \textit{weakly attracted} to a line(path) $\beta$ under the action of $f_{\#}$ if $f_{\#}^k(\gamma)$ weakly converges to $\beta$. 
the \textit{accumulation set} of  $\gamma$ is the set of lines  $l\in \mathcal{B}(G)$ that are elements of the weak closure of a ray $\gamma$ in $G$. This is the set of lines $l$ such that every finite subpath of $l$
occurs infinitely many times as a subpath $\gamma$. Similarly, the weak accumulation set of some point  $\xi\in\partial\mathbb{F}$ is the set of lines in the weak closure of any of the asymptotic rays in its equivalence class.	
\subsection{Free factor systems and malnormal subgroup systems.}
A  finite collection 
$\mathcal{K} = \{[K_1], [K_2], .... ,[K_s]\}$  of conjugacy classes of nontrivial, finite rank subgroups $K_s<\F$ is called a \textit{subgroup system}. 
A  subgroup $K<\F$ is \emph{malnormal} if $xK_sx^{-1}\cap K$ is trivial for all $x\in \F-K$.  A subgroup system is  called \textit{malnormal} if each of its subgroups $[K_s]\in \calK$ is malnormal and  for all $[K_s],[K_t]\in \mathcal{K}$ if $ [K_s] \cap [K_t]$ is non-trivial then $s=t$. 
Given two malnormal subgroup systems $\mathcal{K}, \mathcal{K}'$ we define a partial ordering  $\sqsubset$ 
 on the set of subgroup systems by $\mathcal{K}\sqsubset \mathcal{K}'$ if for each conjugacy class of subgroup $[K]\in \mathcal{K}$ there exists some conjugacy class of subgroup $[K']\in \mathcal{K}'$ such that $K < K'$.

 Given a finite collection $\{K_1, K_2,.....,K_s\}$ of subgroups of $\F$ , we say that this collection determines a \textit{free factorization} of $\F$ if 
$\F = K_1 * K_2 * .....* K_s$.
A \emph{free factor system} $\mathcal{F}:=\{[F_1], [F_2],.... [F_p]\}$ of $\F$ is a finite collection of conjugacy classes of subgroups such that there is a free factorization of $\F$ of the form
$\F = F_1 * F_2 * ....*F_p* B$, where $B$ is some (possibly trivial) finite rank subgroup of $\F$ . Every free factor system is a malnormal subgroup system.

 A malnormal subgroup system $\mathcal{K}$ \textit{carries} a conjugacy class $[c]\in \F$ if there exists some $[K]\in\mathcal{K}$ such that $c\in K$. We say that $\mathcal{K}$ carries a line $\gamma$ if one of the following equivalent conditions hold:
	\begin{enumerate}
 		\item $\gamma$ is the weak limit of a sequence of conjugacy classes carried by $\mathcal{K}$.
 		\item There exists some $[K]\in \mathcal{K}$ and a lift $\widetilde{\gamma}$ of $\gamma$ so that the endpoints of 		$\widetilde{\gamma}$ are in $\partial K$.
	\end{enumerate}

 For any marked graph $G$ and any subgraph $H \subset G$, the fundamental groups of the noncontractible components of $H$ form a free factor system, denoted by $[\pi_1(H)]$. Every free factor system $\mathcal{F}$ can be realized as $[\pi_1(H)]$ for some nontrivial core subgraph H of some marked graph $G$. An equivalent way of saying that a line or circuit  $\gamma$ is carried by $\mathcal{F}$ is that for any marked graph $G$ and a subgraph $H \subset G $ with $[\pi_1(H)]=\mathcal{F}$, the realization of $\gamma$ in $G$ is contained in $H$.

We have the following fact: 
 \begin{lemma}\cite[Fact 1.8]{HM-20}
	 Given a subgroup system $\mathcal{K}$ the set of lines/circuits carried by $\mathcal{K}$ is a closed set in the weak topology.
\end{lemma}
 As a consequence, given a malnormal subgroup system $\mathcal{K}$ and a sequence of lines / circuits $\{\gamma_n\}$, if  $\mathcal{K}$ carries every weak limit of every subsequence of $\{\gamma_n\}$, then $\gamma_n$ is carried by $\mathcal{K}$ for all sufficiently large $n$ (\cite[Lemma 1.11]{HM-20}).

 Given two free factor systems, an extension of the notion of ``intersection '' of subgroups, extended to free factor systems is called `` meet'' of free factors:
	\begin{lemma}[\cite{BFH-00}, Section 2.6]\label{fill} Every collection $\{\mathcal{F}_i\}$ of free factor systems has a well-defined meet $\wedge\{\mathcal{F}_i\} $, which is the unique maximal free factor system $\mathcal{F}$ such that $\mathcal{F}\sqsubset \mathcal{F}_i$ for all $i$. Moreover,
 for any free factor $F< \F$ we have $[F]\in \wedge\{\mathcal{F}_i\}$ if and only if there exists an indexed collection of subgroups $\{A_i\}_{i\in I}$ such that $[A_i]\in \mathcal{F}_i$ for each $i$ and $F=\bigcap_{i\in I} A_i$.
\end{lemma}

The \textit{free factor support}  $\mathcal{F}_{supp}(B)$ of a set of lines $B$ in $\mathcal{B}$ is defined as the meet of all free factor systems that carries $B$ (\cite{BFH-00}). If $B$ is a single line then $\mathcal{F}_{supp}(B)$ is
single free factor. We say that a set of lines $B$ is \textit{filling} if $\mathcal{F}_{supp}(B)=\F$.

\subsection{Attracting laminations and nonattracting subgroup systems:}

For any marked graph $G$, there is a  natural identification $\mathcal{B}\approx \mathcal{B}(G)$ which induces a bijection between the closed subsets of $\mathcal{B(G)}$ of $\mathcal{B}$. A \textit{lamination} $\Lambda$ , is a closed
subset of any of these two sets. Given a lamination $\Lambda\subset \mathcal{B}$ we look at the corresponding lamination in $\mathcal{B}(G)$ as the
realization of $\Lambda$ in $G$. An element $\lambda\in \Lambda$ is called a \textit{leaf} of the lamination.
	
 A lamination $\Lambda$ is called an \textit{attracting lamination} for $\phi$ if it is the weak closure of a line $\ell$ (called the \textit{generic leaf of $\Lambda$}) satisfying:
\begin{itemize}
 \item $\ell$ is bi-recurrent leaf of $\Lambda$: every finite subpath of $\ell$ occurs infinitely many times as subpath in both directions.
\item There is a neighborhood $V^+$ such that every line in $V^+$ is weakly attracted to $\ell$ in the weak topology. $V^+$ is called an \textit{attracting neighborhood} of $\ell$.  (see \cite[Definition 3.1.1]{BFH-00})
\item no lift $\widetilde{\ell}\in \mathcal{B}$ of $\ell$ is the axis of a generator of a rank 1 free factor of $\F$ .
\end{itemize}
Attracting neighborhoods of $\Lambda$ are defined by choosing sufficiently long segments of generic leaves of $\Lambda$. Note that if $V^+$ is an attracting neighborhood, then $\phi(V^+)\subset V^+$. 
	
  Associated to each $\phi\in \out$ is a finite set $\mathcal{L}^+(\phi)$ of laminations,  called the set of \textit{attracting laminations} of $\phi$ (\cite{BFH-00}) . Similarly we define  the set of attracting laminations (or the set of \emph{repelling laminations}) $\mathcal{L}^-(\phi)$ of $\phi^{-1}$ (or of $\phi$). Both $\mathcal{L}^+(\phi)$ and $\mathcal{L}^-(\phi)$ are $\phi$-invariant (\cite[Lemma 3.1.13, Lemma 3.1.6]{BFH-00}). Given a relative train-track map $f: G\to G$ representing $\phi$, then there is a bijection between  the set $\mathcal{L}^+(\phi)$ and the set of exponentially growing strata in $G$ which is determined by the height of a generic leaf of $\Lambda^+\in\mathcal{L}^+(\phi)$ (\cite[Section 3]{BFH-00}).

 Free factor support of an element $\Lambda$ of $\mathcal{L}^+(\phi)$ or $\mathcal{L}^-(\phi)$ is the smallest free factor system that carries each leaf of  $\Lambda$, denoted  $\mathcal{F}_{supp}(\Lambda)$ ([\cite{BFH-00}, Section 3.2]). Free factor support of $\Lambda^+\in\mathcal{L}^+(\phi)$ is equal to free factor support of a generic leaf of $\Lambda^+$ (\cite[Corollary 2.6.5, Corollary 3.1.11]{BFH-00}). If $\Lambda_1, \Lambda_2 \in \mathcal{L}^+(\phi)$ then $\Lambda_1 = \Lambda_2 \Leftrightarrow  \mathcal{F}_{supp}(\Lambda_1) = \mathcal{F}_{supp}(\Lambda_2)$ (\cite[Fact 1.14]{HM-20}).
	
An element of $\mathcal{L}^+(\phi)$ is dual to an element of $\mathcal{L}^-(\phi)$ if these two elements  have the same free factor support. This imposes a bijection between these two sets \cite[Lemma 3.2.4]{BFH-00}.

 A line/circuit $\gamma$ is said to be \emph{weakly attracted} to $\Lambda_1\in\mathcal{L}^+(\phi)$ if $\gamma$ is weakly attracted to some (hence every) generic leaf of $\Lambda$ under action of $\phi$. No leaf of any element of $\mathcal{L}^-(\phi)$ is ever attracted to any leaf of any element of $\mathcal{L}^+(\phi)$. No Nielsen path is weakly attracted to either an element of $\mathcal{L}^+(\phi)$ or an element of $\mathcal{L}^-(\phi)$.

 An attracting lamination $\Lambda^+$ and a repelling lamination $\Lambda^-$ of $\phi$ cannot have leaves which are asymptotic, for otherwise it would violate the fact that no leaf of $\Lambda^-$ is weakly attracted to $\Lambda^+$. This allows us to choose sufficiently long subpaths of generic leaves of $\Lambda^+, \Lambda^-$ and construct attracting and repelling neighborhoods $V^+, V^-$ so that $V^+\cap V^-=\emptyset$.

An element of $\Lambda\in \mathcal{L}^+(\phi)$ is said to be \emph{topmost} if there does not exists any $\Lambda_j\in\mathcal{L}^+(\phi)$ such that $\Lambda \subset \Lambda_j$.

 Elements of $\mathcal{L}^+(\phi)$ can be divided into two distinct classes. If there exists a finite, $\phi$-invariant collection of  distinct, nontrivial  conjugacy classes $\mathcal{C} = \{[c_1], [c_2], \ldots [c_s]\}$ such that $\mathcal{F}_{supp}(\mathcal{C}) = \mathcal{F}_{supp}(\Lambda^+)$, then $\Lambda^+ \in \mathcal{L}^+(\phi)$ is said to be \emph{geometric}(\cite[Definition 2.19]{HM-20}). It is said to be \emph{nongeometric} otherwise.


The \emph{nonattracting subgroup system} of an attracting lamination was first introduced by Bestvina-Feighn-Handel in \cite{BFH-00} and later explored in great details by Handel-Mosher in \cite[Part III, pp-192-202]{HM-20}. Nonattracting subgroup system records the information about the lines and circuits which are not attracted to the lamination.  We will list some of the properties which are central to our proofs.
\begin{lemma}(\cite{HM-20}- Lemma 1.5, 1.6)\label{lem:na}
\label{NAS} Let $\phi\in\out$ and $\Lambda^+\in \mathcal{L}^+(\phi)$ be an attracting lamination such that $\phi(\Lambda^+) = \Lambda^+$. Then there exists a subgroup system $\mathcal{A}_{na}(\Lambda^+)$ such that:
 \begin{enumerate}
  \item $\mathcal{A}_{na}(\Lambda^+)$ is a malnormal subgroup system and the set of lines carried by $\mathcal{A}_{na}(\Lambda^+)$ is closed in the weak topology.
  \item A conjugacy class $[c]$ is not attracted to $\Lambda^+$ if and only if it is carried by $\mathcal{A}_{na}(\Lambda^+)$.
  \item If each conjugacy class of  a finite rank subgroup $B < \F$ is not weakly attracted to $\Lambda^+_\phi$, then there exists some
  $A < \F$ such that $B< A$ and $[A]\in \mathcal{A}_{na}(\Lambda^+_\phi)$.
  \item  If $\Lambda^-$ and  $\Lambda^+$ are dual to each other, we have $\mathcal{A}_{na}(\Lambda^+)= \mathcal{A}_{na}(\Lambda^-)$.
  \item $\mathcal{A}_{na}(\Lambda^+)$ is a free factor system if and only if $\Lambda^+$ is not geometric.
  \item If $\{\gamma_n\}_{n\in\mathbb{N}}$ is a sequence of lines such that every weak limit of every subsequence of $\{\gamma_n\}$ is carried by $\mathcal{A}_{na}(\Lambda^+)$ then for all sufficiently large $n$, $\{\gamma_n\}$ is carried by $\mathcal{A}_{na}(\Lambda^+)$.  
  
 \end{enumerate}
\end{lemma}
A pseudo-Anosov on a surface with non empty boundary induces a fully irreducible automorphism called  \emph{geometric} fully irreducible. Geometricity was relativized by Bestvina-Feighn Handel \cite{BFH-00} by extending it to so that it is a property of an
EG stratum  of a relative train track map.
The definition of geometricity is expressed in terms of the existence of what we call
a ``geometric model" (See \cite{HM-20} for details of constructions of geometric models.)

When $\Lambda^+$ is geometric, the geometric model gives us the structure of the nonattracting subgroup system (\cite[Definition 2.1, Lemma 2.5]{HM-20}). In particular, there is a unique closed indivisible Nielsen path $\rho_r$ of height $r$. (\cite[Fact 1.42]{HM-20})

 The facts listed above show that when $\Lambda^+$ is geometric, $\mathcal{A}_{na}(\Lambda^+) = \mathcal{F} \cup\{[F_m]\}$, where $\mathcal{F}$ is some free-factor system such that $[\pi_1 G_{r-1}] \sqsubset\mathcal{F}$. Moreover,  $\phi$ restricted to $F_m$ has polynomial growth and if $[c]$ is the conjugacy class determined by $\rho_r$, then $[c]$ is carried by $[F_m]$. We shall refer to the component $[F_m]$ a \emph{geometric component} of $\mathcal{A}_{na}(\Lambda^+)$. In case $H_r$ is a top stratum, $\mathcal{A}_{na}(\Lambda^+) = \mathcal{F} \cup\{[\langle c \rangle]\}$.

\subsection{Completely split improved relative train track (CT) maps}\label{sec:CT}
A \textit{splitting} of a line, a path or a circuit $\alpha$ is a concatenation $ \ldots\alpha_0\alpha_1 \ldots\alpha_k\ldots $ of subpaths of $\alpha$ in $G$ 
 such that for all $i\geq 1$, $f^i_\#(\alpha) =  \ldots f^i_\#(\alpha_0)f^i_\#(\alpha_1)\ldots f^i_\#(\alpha_k)\ldots$, for a relative train track map $f:G\rightarrow G$.
 The subpath $\alpha_i$' s in a splitting are called \emph{terms} or \emph{components} of the splitting of $\alpha$. The notation $\alpha\cdot\beta$ will denote a splitting and $\alpha\beta$ will denote a concatenation of nontrivial paths $\alpha, \beta$.

  A splitting $\alpha_1\cdot\alpha_2\cdots\cdot \alpha_k$ of a  non trivial path or circuit $\alpha$ is  a \textit{complete splitting} if each component  $\alpha_i$ is either a single edge in an irreducible stratum, an indivisible Nielsen path, an
exceptional path (see \cite[Definition 4.1]{FH-11}) or a \emph{taken} connecting path in a zero stratum (see \cite[Definition 4.4]{FH-11}). 

\emph{Completely split improved relative train track (CT) maps } are topological representatives with particularly nice properties. In particular, for each edge $E$ in each irreducible stratum, the
path $f_\#(E)$ is completely split. Moreover, for each taken connecting path $\alpha$ in each zero stratum, the path $f_\#(\alpha)$ is completely split. CT maps are guaranteed to exist for \textit{rotationless} (see Definition 3.13 \cite{FH-11}) outer automorphisms, as has been shown in the following Theorem 4.28 from \cite{FH-11}.
 \begin{lemma}
  For each rotationless $\phi\in \out$ and each increasing sequence $\mathcal{F}$ of $\phi$-invariant free factor systems, there exists a CT map $f:G\rightarrow G$ that is a topological
  representative for $\phi$ and $f$ realizes $\mathcal{F}$.
 \end{lemma}

Feighn-Handel \cite{FH-11} showed that there exists some $k>0$ such that, given any $\phi\in\out$, $\phi^k$ is rotationless. So given any outer automorphism $\phi$, some finite power of $\phi$ has a completely split improved relative train track representative. In the body of our work, we will use  CT's defined in the  work of Feighn-Handel in \cite{FH-11}. In particular, we will use the properties given in  \cite[Definition 1.29]{HM-20} assume all exponentially growing outer automorphisms are rotationless.


 \subsection{Critical Constant}\label{sec:assump}

Let $\phi\in\out$ be exponentially growing  and $f\co G \to G$ be a CT map representing $\phi$. Below we will adapt a \emph{bounded cancellation} notion coming from  Bestvina-Feighn-Handel's bounded cancellation lemma for train-track representatives of automorphisms of $\F$ (\cite{BFH-97}); which is inspired by Cooper's same named lemma (\cite{Co-87}).  To summarize, if we have a path in $G$ which has some $r-$legal ``central''  subsegment of length greater than the
    critical constant given below, then this segment is protected by the bounded cancellation lemma and the
    length of this segment grows exponentially under iteration.

\begin{definition}
Let  $H_r$ be an exponentially growing stratum with associated Perron-Frobenius eigenvalue $\lambda_r$ and let $BCC(f)$ denote the bounded cancellation constant for $f$. Then, the number

  \[\frac{2BCC(f)}{\lambda_r-1}\]

  is called  the \emph{critical constant} for $H_r$.

\end{definition}
  It is easy to see  that for every number $C>0$ that exceeds the
  critical constant, there is some $1\geq\mu>0$ such that if $\alpha\beta\gamma$ is a concatenation of $r-$legal paths where
   $\beta $ is some $r-$legal segment of length $\geq C$, then the $r-$legal leaf segment of
   $f^k_\#(\alpha\beta\gamma)$ corresponding to $\beta$ has length  at least $\mu\lambda^k{|\beta|}_{H_r}$  (see \cite[pp 219]{BFH-97}). Here ${|\beta|}_{H_r}$ is the length of an edge-path $\beta \subset G$ is the number of the edges of $\beta$ that remain only in $H_r$. 

For the rest of this paper, let $C$ be a number larger than the maximum of all critical constants corresponding to EG strata of $f\co G\rightarrow  G$.

\subsection {Admissible subgroup systems and relative hyperbolicity} \label{sec:K}

 Given a group $G$ and a collection $\{K_\alpha\}$ of subgroups $K_{\alpha} < G$, we obtain the \emph{coned-off Cayley graph of $G$} or
  the \emph{electrocuted $G$} relative to the collection $\{K_\alpha\}$ by assigning a vertex  $v_\alpha$ for each left coset of $K_\alpha$ on the
  Cayley graph of $G$  such that each point of a left coset of 
  $K_\alpha$ is joined to (or coned-off at) $v_\alpha$ by an edge of length $1/2$. The resulting metric space is
  denoted by $(\widehat{G}, {|\cdot|}_{el})$.

 A group $G$ is said to be (weakly) relatively hyperbolic relative to the collection of subgroups
  $\{K_\alpha\}$ if $\widehat{G}$ is a $\delta-$hyperbolic metric space. 
  $G$ is said to be strongly hyperbolic relative to the collection $\{K_\alpha\}$ if the coned-off
  space $\widehat{G}$ is weakly hyperbolic relative to $\{K_\alpha\}$ and it satisfies a certain
  \textit{bounded coset penetration} property (see \cite{Fa-98}). In this paper, when we say ``relative hyperbolicity" we always mean ``strong relative hyperbolicity".

\begin{definition}

Given   $\phi\in \out$, 
let $\mathcal{K} = \{ [K_1], [K_2], \ldots , [K_p]\}$   be a subgroup system and $\mathcal{L}^{+}_{\mathcal{K}}(\phi)$ (respectively, $\mathcal{L}^{-}_{\mathcal{K}}(\phi)$) denote the collection of attracting (respectively, repelling) laminations of $\phi$ whose generic leaves are not carried by $\mathcal{K}$. Assume,
	\begin{enumerate} \label{SA}
 \item $\mathcal{L}^{+}_{\mathcal{K}}(\phi), \mathcal{L}^{-}_{\mathcal{K}}(\phi)$ are both nonempty, 
		\item $\mathcal{K}$ is a malnormal subgroup system, 
		\item $\phi(K_s) = K_s$ for each $s\in \{1, \cdots p\}$, 
		\item Let $V^+$ denote the union of attracting neighborhoods of elements of $\mathcal{L}^+_{\mathcal{K}}(\phi)$ defined by generic leaf segments of length $\geq 2C$. Define $V^-$ similarly for $\mathcal{L}^-_{\mathcal{K}}(\phi)$.  By increasing $C$ if necessary,
          \[V^+ \cap V^- = \emptyset\]
		
		\item Every conjugacy class which is not carried by $\mathcal{K}$ is weakly attracted to some element of $\mathcal{L}^+_{\mathcal{K}}(\phi)$.
	
	\end{enumerate}

We will call subgroup systems satisfying the properties above an \emph{admissible subgroup system for $\phi$} (where $\phi\in \out$ is  exponentially growing). If we are given some finitely generated group $Q = \langle \phi_1, \phi_2, \ldots, \phi_k \rangle$ we say that $\mathcal{K}$ is an \emph{admissible subgroup system for $Q$} if $\mathcal{K}$ is an admissible subgroup system for each $\phi_i$. We will simply write ``admissible subgroup system" when the context is clear. 
\end{definition}

In our previous work, we proved,

\begin{theorem}\cite{GGrelhyp}\label{main1}
Let $\phi_1, \ldots , \phi_k $ be a collection of exponentially growing outer automorphisms of $\mathbb F$ such that $\phi_i, \phi_j$ do not have a common power whenever $i\neq j$. Then the following statements are equivalent:
\begin{enumerate}
    \item There exists a subgroup system $\mathcal{K} = \{[K_1], \ldots , [K_p]\}$ which is admissible for each $\phi_i$ and for which $\mathcal{L}_\mathcal{K}^\pm(\phi_i) \cap \mathcal{L}_\mathcal{K}^\pm(\phi_j) = \emptyset$ whenever $i\neq j$. 
    \item {For every free group $Q$ generated by sufficiently high powers of $\phi_i$'s, $\F\rtimes\widehat{Q}$ is hyperbolic relative to a collection $\{K_s \rtimes \widehat{Q}_s\}_{s= 1}^p$ where $\widehat{Q}_s$ is a lift of $Q$ that fixes $K_s$. }
\end{enumerate}
\end{theorem}

\section{Detecting Obstructions to (relative) hyperbolicity on the complex of free factors}\label{sec:3}

\textbf{Why high powers?:} Throughout this paper we have worked with a subgroup $Q$ which becomes a free group after taking high powers of outer automorphisms, which satisfy certain conditions. We wish to clarify the reasons for resorting to high powers. Firstly, there does not exist a good (for our purposes) hyperbolic metric space where hyperbolic outer automorphisms act loxodromically. The second necessity comes from a technicality, which we describe via the following example:

Let $\phi$ be the outer automorphism class of the automorphism corresponding to $ \Phi\co a\mapsto ad, b\mapsto a, c\mapsto b, d\mapsto c$ and let $\psi$ be the outer automorphism class of the automorphism corresponding to $\Psi\co a\mapsto ac, b\mapsto a, c\mapsto b, d\mapsto db$. 
Observe that both $\phi$ and $\psi$ are positive automorphisms and one can get a relative train-track map (in fact a CT map) on the standard rose with four petals.  $\phi$ has only one strata (the full graph) and $\psi$ has an NEG strata (namely $d$) sitting on top of an EG strata (namely $\{a, b, c\}$), but the NEG edge $d$ grows exponentially under iteration. In both cases, we have exactly one EG strata and the corresponding attracting lamination has a trivial nonattracting subgroup system by design.

Thus $\phi, \psi$ are both atoroidal outer automorphisms (since nonattracting subgroup system for the unique attracting lamination is trivial in both cases). Also, they have disjoint laminations since by design the free factor support of the attracting lamination of $\psi$ is $[\langle a, b, c \rangle ]$ whereas $\phi$ is fully irreducible, hence the associated lamination fills. These automorphisms are sufficiently different and have empty sinks.

If one looks at the extension group $E_Q$ associated to the group $Q=\langle \phi, \psi \rangle$, then it follows that $\Phi^{-1}\Psi$ fixes both $b, c$ and hence the group $E_Q$ has multiple copies of $\mathbb{Z} \oplus \mathbb{Z}$, so $E_Q$ cannot be hyperbolic. Now consider $H = \langle \Phi^{-1}\Psi, b, c\rangle$. If $E_Q$ were relatively hyperbolic, then $H$ would be conjugate into some peripheral subgroup (call it $P$) in this relatively hyperbolic structure. But  $b\in \Phi H \Phi^{-1}$ together with malnormality implies that $H$ and $\Phi H \Phi^{-1}$ are both contained in $P$. Hence $\Phi b \Phi^{-1} \in P \implies a \in P$.  Applying the same argument again, $\Phi a \Phi^{-1} \in P \implies ad \in P \implies d\in P$. As a result, $ \F < P$. Hence, $E_Q$  cannot be relatively hyperbolic.  

However, as we shall see later if the group $Q$, whose generators are sufficiently different with empty corresponding sinks  is generated instead by \emph{sufficiently high powers of} $\phi $ and $\psi$, its extension $E_Q$ is hyperbolic --which is something one expects. This example highlights a tractable but key issue due to which we need to pass to high powers throughout this paper.


\subsection{Sufficiently different outer automorphisms and their dynamics on $\FF$}\label{sec:suffdiff}

In this section we will let our  exponentially growing automorphisms to have fixed points, but require those fixed points to be sufficiently apart  from each other in $\FF$. We begin our  analysis on $\FF$ and  combine dynamics of automorphisms on $\FF$ with the dynamics on the  set  of laminations via weak attraction theory.
\begin{definition}

A collection of outer automorphisms will be called \emph{sufficiently different} if,
\begin{itemize}
\item No pair of the automorphisms  have a common power (hence they do not generate a virtually cyclic  (elementary) subgroup of $\out$),  and 
\item Fixed vertices (if any) of these outer automorphisms on $\FF$ are distance at least $2$ from each other in $\FF$.
\end{itemize}
\end{definition} Given that $\FF$ has infinite diameter, it is natural to expect that two randomly chosen exponentially growing outer automorphisms will be sufficiently different. Hence, we study their dynamics in relation to one another.

\begin{lemma}\label{pfifbf}
    Let $\phi_1, \ldots , \phi_k$  be a collection of exponentially growing outer automorphisms of $\mathbb F$ that are sufficiently different. Then, 
    \begin{enumerate}
        \item No generic leaf of any attracting or repelling lamination of a $\phi_i$ is carried by the nonattracting subgroup system of any attracting or repelling lamination of $\phi_j$, for $i\neq j$.
        \item  There exists some $M$ such that whenever $m_i \geq M$, for all $i$, $\langle \phi_1^{m_1}, \cdots \phi_k^{m_k}\rangle$ is a free group.
    \end{enumerate}

\end{lemma}
\begin{proof}
    If no $\phi_i$ fixes any vertex in the free factor complex, then they are all fully irreducible and the attracting and repelling laminations associated to them  fill $\F$ (See Lemma \ref{fill} and the paragraph afterwards). The result then follows from an induction argument applied to  \cite[Proposition 3.7]{BFH-97}.

    Suppose that some $\phi_i$ fixes some vertex in $\FF$.
    Let $\Lambda^+_j$ be an attracting lamination of $\phi_j$, $j\neq i$. If $\mathcal{F}_{supp}(\Lambda^+_j)$ fills then we are done as in this case no generic leaf of $\Lambda^+_j$ can be carried by the nonattracting subgroup system of any attracting or repelling lamination of $\phi_i$. 
    Suppose then that $\mathcal{F}_{supp}(\Lambda^+_j)$ is proper and $\Lambda^+_j$ is carried by $\mathcal{A}_{na}(\Lambda^+_i)$ for  some attracting lamination $\Lambda^+_i$ of $\phi_i$. This would imply that $\mathcal{F}_{supp}(\Lambda^+_j)$ is carried by a free factor component $[B]\in\mathcal{A}_{na}(\Lambda^+_i)$. This violates the distance requirement of being sufficiently different. So the result follows.
    
    To prove (2), we claim that  there exists an attracting (resp. repelling) lamination $\Lambda^+_i$ (resp. $\Lambda^-_i$) of $\phi_i$ and an attracting (resp. repelling) lamination $\Lambda^+_j$ (resp. $\Lambda^-_j$) of $\phi_j$ such that no generic leaf of $\Lambda^+_i$ or $\Lambda^-_i$  is asymptotic to any generic leaf of $\Lambda^+_j$ or $\Lambda^-_j$. Assuming the claim to be true we are done by  Corollary 2.17 \cite{HM-20} and applying induction to \cite[Lemma 3.4.2]{BFH-00} (Detecting $\F_2$ via laminations).
    
    To prove the claim, suppose first that there exists some attracting lamination $\Lambda$ of $\phi_i$ such that $\mathcal{F}_{supp}(\Lambda)$ is proper. If some generic leaf of $\Lambda$ is asymptotic to some generic leaf of an attracting or repelling lamination of $\phi_j$ then birecurrency of generic leaves implies that both ends of a leaf  would have the same height, and as a result these lines would have the same free factor support. Hence,  $\mathcal{F}_{supp}(\Lambda)$ will be invariant under both $\phi_i$ and $\phi_j$, thus violating the distance requirement in $\FF$ between their fixed vertices. So the claim is proved in the case of existence of any attracting or repelling lamination for either $\phi_i$ or $\phi_j$ with a proper free factor support.  Suppose then neither $\phi_i$ nor  $\phi_j$ have an attracting lamination with a proper free factor support. This implies that there exist unique attracting laminations $\Lambda^+_i, \Lambda^+_j$ for $\phi_i, \phi_j$ respectively that fill $\F$. The filling condition implies that  $\Lambda^+_i, \Lambda^+_j$ are necessarily topmost. If $\Lambda^+_i, \Lambda^+_j$ have asymptotic generic leaves, say $\ell_i, \ell_j$ respectively then $\ell_i$ cannot be carried by the nonattracting subgroup system of $\Lambda^+_j$. Also, asymptoticity will imply that $\ell_i$ cannot be weakly attracted to $\Lambda^-_j$. Hence by the weak attraction theorem (\cite[Theorem 6.0.1]{BFH-00}) $\ell_i$ must be a generic leaf of $\Lambda^+_j$ . But by being a generic leaf, closure of $\ell_i$ is all of $\Lambda^+_j$ (and it is also all of  $\Lambda^+_i$ by the nature of $\ell_i$). Hence $\Lambda^+_i = \Lambda^+_j$. 
    As a result, in this case the following proposition completes the proof of our claim by giving us a contradiction with the hypothesis that no pair $\phi_i, \phi_j$  have common power.
\end{proof}
Following is a technical result which identifies dynamical conditions for detecting when a subgroup of $\out$ is virtually cyclic.
\begin{proposition}\label{fillinglam}
    Let $\phi_1, \ldots , \phi_k$  be  a collection of exponentially growing automorphisms of $\out$  for which following conditions are satisfied:
    \begin{enumerate}
        \item No pair $\phi_i, \phi_j$ with $i\neq j$ have a common invariant free factor.
        \item There exists a common attracting lamination $\Lambda^+$ such that $\Lambda^+$ fills.
    \end{enumerate}
    Then all $\phi_i$ are fully irreducibles with common powers.
\end{proposition}
\begin{proof}
    We will prove the theorem for $k=2$ to avoid extra notation as the general case is the same.
    Consider the group $\mathcal{H} = \langle \phi_1, \phi_2\rangle < \out$. Then $\mathcal{H} < \text{Stab}(\Lambda^+)$. If either $\phi_1$ or $\phi_2$ is fully irreducible, then $ \text{Stab}(\Lambda^+)$ is virtually cyclic by \cite[Theorem 2.14]{BFH-97}. So $\mathcal{H}$ must be virtually cyclic and  hence $\phi_1, \phi_2$ are both fully irreducible and must have common powers. 
    
    Suppose then that both $\phi_1, \phi_2$ are reducible. Then the group $\mathcal{H}$ is fully irreducible, since $\phi_1, \phi_2$ do not have any common invariant free factor. By applying \cite[Theorem C\' - page 2]{HM-20} with $\mathcal{F} = \emptyset$, we see that there exists a fully irreducible $\xi\in\mathcal{H} < \text{Stab}(\Lambda^+)$. Let $\Lambda^+_\xi, \Lambda^-_\xi$ be the unique attracting and repelling laminations of $\xi$ and $\ell^+$ be a generic leaf of $\Lambda^+_\xi$. Then $\Lambda^+_\xi$ and $\Lambda^-_\xi$ both fill $\F$. If $\Lambda^+ = \Lambda^+_\xi $ or $\Lambda^-_\xi$, then we are reduced to the earlier case. So suppose $\Lambda^+$ is distinct from $\Lambda^+_\xi$ and $\Lambda^-_\xi$. Choose a train-track map $f: G\to G$ for $\xi$. Let $V^+_\xi$ be some attracting neighborhood of $\Lambda^+_\xi$, which is described by some sufficiently long generic leaf segment of $\ell^+$. Similarly choose a repelling neighborhood  $V^-_\xi$. Now, by our assumption, if $\gamma^+$ is a generic leaf of $\Lambda^+$, then $\gamma^+$ fills, and it is not a generic leaf of $\Lambda^-_\xi$. Then by the weak attraction theorem Corollary 2.17 \cite{HM-20} together with \cite[Theorem G]{HM-20}, there exists some $k>0$ such that $f^k_\#(\gamma^+) \in V^+_\xi$ . Since $\xi\in \text{Stab}(\Lambda^+)$ and $f^k_\#(\gamma^+)$ is always birecurrent,  $f^k_\#(\gamma^+)\in \Lambda^+$ is a birecurrent leaf. Hence we conclude that some birecurrent leaf of $\Lambda^+$ contains the subpath of $\ell^+$ which described $V^+_\xi$. As $\gamma^+$ was chosen to be generic leaf, we have $\overline{\gamma^+} = \Lambda^+$. Hence the chosen subpath of $\ell^+$ is also a subpath of $\gamma^+$. Since this subpath was chosen arbitrarily, we deduce that $\ell^+$ is in the closure of generic leaves of $\Lambda^+$, implying that $\Lambda^+_\xi\subset \Lambda^+$. 
    
    Now consider the dual lamination $\Lambda^-$ of $\Lambda^+$. Since $\ell^+$ is a filling line, by the weak attraction theorem of \cite[Theorem 6.0.1]{BFH-00} either $\ell^+$ is weakly attracted to $\Lambda^-$ or $\ell^+$ is a generic leaf of $\Lambda^+$. Since being weakly attracted to a lamination is an open condition and every subpath of $\ell^+$ is a subpath of some generic leaf of $\Lambda^+$, the first option is not possible. Hence $\ell^+$ must be a generic leaf of $\Lambda^+$, implying that $\Lambda^+ = \Lambda^+_\xi$, which is a contradiction to our assumption that they are distinct. Therefore $\Lambda^+ = \Lambda^+_\xi$ or $\Lambda^+ = \Lambda^-_\xi$ and we are done by the first case.
\end{proof}

\begin{remark}
  When $\Lambda^\pm_1$ and $\Lambda^\pm_2$ are both geometric and $\phi_1, \phi_2$ satisfy the hypothesis of Proposition \ref{pfifbf}, it may seem that no additional restriction is placed on the conjugacy class representing the unique indivisible Nielsen path associated to the laminations.  However, we shall see in Lemma \ref{rhconditions} that the hypothesis that the fixed vertex sets of $\phi_1$ and $\phi_2$ are of distance at least $2$ in $\FF$ puts some very strong restrictions on them.
\end{remark}

\subsection {Cusp preserving relative hyperbolicity} \label{sec:cusprelhyp}

Given a short exact sequence of finitely generated groups $1 \to \F \to E \to Q \to 1$, we say that $E$ has a \textbf{\emph{cusp-preserving}} relatively hyperbolic structure if there exists a collection of finitely generated subgroups $\{H_i\}$ of $\F$ such that $Q$ preserves conjugacy class of each $H_i$ and $E$ is (strongly) hyperbolic relative to the collection $\{N_E(H_i)\}$ (where $N_E(H_i)$ is the normalizer of $H_i$ in $E$) (see \cite[Theorem 5.1]{GGrelhyp} for motivation behind this property).

Recall that  an outer automorphism $\phi$ is said to be \emph{polynomially growing} if for each automorphism $\Phi\in \Aut(\F)$ representing $\phi$ and each $c\in \F$ the cyclically reduced word length of $\phi^i(c)$ is bounded above by a polynomial of $i$.

A subgroup $Q< \out$ is called \emph{upper polynomially growing (UPG)} if each $\phi\in Q$ is polynomially growing and $Q$ has unipotent image in $GL_n(\mathbb Z)$. If $Q<\out$ is UPG we will call every $\phi\in Q$ also UPG.

Recall that if $\Lambda_1, \Lambda_2$ are attracting laminations for $\phi$ with nonattracting subgroup systems $\mathcal{K}_1, \mathcal{K}_2$ respectively, then we define the \emph{meet} (Lemma \ref{fill}) $\mathcal{K}_1\bigwedge\mathcal{K}_2$  as follows:
\[ \mathcal{K}_1 \bigwedge \mathcal{K}_2  = \{ [A \cap B^w] : [A]\in \mathcal{K}_1, [B] \in \mathcal{K}_2, w\in \F, A \cap B^w \neq \{id\} \}\]
where $B^w$ denotes the group $w^{-1} B w$, with  $w\in \F$.

\begin{definition}(Sink of an automorphism) \cite[Section 4.3]{GGrelhyp}\label{sink} Given any exponentially growing outer automorphism $\phi$ of $\F$, consider the full list of attracting laminations $\{\Lambda_q\}_{q = 1}^r$ and let
\[\mathcal{K}^*_\phi = \mathcal{K}_1\bigwedge\mathcal{K}_2\bigwedge\ldots\bigwedge \mathcal{K}_r, \,\,\text{where} \,\,\mathcal{K}_q = \mathcal{A}_{na}(\Lambda_q) \,\,\text{for} \,\,q = 1, \ldots,  r.\] If $\phi$ is not exponentially growing define $\mathcal{K}^*_\phi = \{[\F]\}$.  We shall call $\mathcal{K}_\phi^*$  the \emph{nonattracting sink} of $\phi$.

 The nonattracting sink $\mathcal{K}^*_\phi$ of $\phi$ is thus a $\phi-$invariant malnormal subgroup system which carries all conjugacy classes which are either fixed or grow polynomially under iteration by $\phi$. Thus, by construction, the nonattracting sink is designed to trap obstructions to hyperbolicity properties of the extension groups. 
We have,
\end{definition}
\begin{lemma}\label{sinkprop}
    Given an exponentially growing $\phi\in\out$ and its nonattracting sink $\mathcal{K}^*_\phi$, we have
    \begin{enumerate}
        \item The nonattracting sink of $\phi$ is an admissible subgroup system for $\phi$. 
        \item A conjugacy class is carried by $\mathcal{K}^*_\phi$ if and only if it is not weakly attracted to any attracting lamination of $\phi$. 
        \item A conjugacy class is carried by $\mathcal{K}^*_\phi$ if and only if it is either fixed by $\phi$ or grows polynomially under iteration by $\phi$.
        \item $K^*_\phi = \mathcal{F}\cup \{[F_1], \ldots, [F_p]\}$, for some free factor system $\mathcal{F}$ and finite rank $\phi-$invariant subgroups $F_i\in\F$, each of which is a geometric component of the nonattracting subgroup system of some geometric lamination of $\phi$. 
        \item If $\mathcal{K}^*_\phi$ has only infinite cyclic components, then no conjugacy class has nontrivial polynomial growth under iteration by $\phi$. 
    \end{enumerate}
\end{lemma}
\begin{proof}
(1) and (2) follows from \cite[Corollary 4.13]{GGrelhyp}. 
(3) follows from the fact that any conjugacy class grows exponentially under iteration by $\phi$ if and only if it is weakly attracted to some attracting lamination of $\phi$ under iteration. 

(4) follows from the proof of part (1) of \cite[Proposition 4.12]{GGrelhyp}. 

To prove (5) note that if $[w]$ is a conjugacy class that grows polynomially under iteration by $\phi$, then there must be at  least a rank $2$ subgroup which carries both $[w]$ and a twistor to generate at least linear order growth. This proves that if $\mathcal{K}^*_\phi$ has only infinite cyclic components, then  no conjugacy class can have nontrivial polynomial growth. 
\end{proof}

\textbf{Remark:} We point out to the reader that if $K^*_\phi = \mathcal{F}\cup \{[F_1], \ldots, [F_p]\}$ and then each nontrivial $[F_i]$ arises due to presence of a geometric strata and restriction of $\phi$ to each nontrivial $F_i$ is polynomially growing. Moreover, each nontrivial $[F_i]$ carries the conjugacy of the unique closed indivisible Nielsen path associated to a geometric strata. This is an essential ingredient in the proof of \cite[Proposition 4.12]{GGrelhyp}

We now proceed with the following proposition which identifies a sufficient condition for a collection of outer automorphisms to not have a common admissible subgroup system. 
\begin{proposition}\label{nrh}
     Let $\phi_1, \cdots,  \phi_k$ be a collection of exponentially growing automorphisms so  that no pair of elements have a common power. Let $Q$ be any free group generated by some powers of $\phi_i$'s. Let $\mathcal{K}^*_i$ be the nonattracting sink of $\phi_i$ and assume that the following conditions hold:
     \begin{enumerate}
         \item $\mathcal{K}_i^*\neq \emptyset$ for some $i\in \{1,\cdots, k\}$.
         \item If $\mathcal{K}^*_i\neq \emptyset$, then $\mathcal{K}^*_i$ is not carried by $\mathcal{A}_{na}(\Lambda^+_j)$ for any attracting lamination $\Lambda^+_j$ of $\phi_j$ for some $j \neq i$.
     \end{enumerate}
      Then we have the following: 
      \begin{enumerate}[\label = (a)]
          \item There is no admissible subgroup system for $Q$.  
          \item The extension group $\F \rtimes \widehat{Q}$ is not hyperbolic relative to any collection of subgroups so that the cusps are preserved.
      \end{enumerate}
\end{proposition}

\begin{proof}
    Suppose that $(1)$ and $(2)$ hold. Without loss of generality assume that $\mathcal{K}^*_1 \neq \emptyset$. Using item $(3)$ of Lemma \ref{NAS}, condition $(2)$ implies that $\mathcal{K}^*_1$ cannot be carried by any admissible subgroup system for $\phi_j$, for some $j\neq 1$. If $\mathcal{K}$ were an admissible subgroup system for $Q$, then every conjugacy class not carried by $\mathcal{K}$ would  get weakly attracted to some element of $\mathcal{L}^+_\mathcal{K}(\phi_s)$ (hence would grow exponentially under iteration by $\phi_s$), for each $1\leq s\leq k$. Hence conclusion $(a)$ holds.
    
    Now suppose that $\F \rtimes \widehat{Q}$ is hyperbolic relative to some collection of peripheral subgroups. Let $[g]$ be carried by $\mathcal{K}^*_i$. Then $[g]$ is not weakly attracted to any attracting lamination of $\phi_i$ and hence by definition either $[g]$ is fixed by $\phi_i$ or $|{\phi_i}^n_\#([g])|$ grows polynomially. So $g$ must be conjugate to some element in a peripheral subgroup. As a consequence, the peripheral subgroup system in Theorem \ref{main1} is nonempty. Moreover, whenever $\F\rtimes \widehat{Q}$ preserves cusps, the same theorem ensures the existence of an admissible subgroup system, contradicting conclusion (a) above. This completes proof of (b). 
\end{proof}

\begin{corollary}\label{rank1ff}
    Let $\phi_1, \cdots,  \phi_k$ be a collection of exponentially growing and sufficiently different outer automorphisms of $\F$. Then if some  $\phi_i$ fixes the conjugacy class of some rank $1$ free factor $[\langle g \rangle]$, then $[g]$  can not be carried by a free factor component of the nonattracting subgroup system of any attracting lamination of $\phi_j$,  for any $ \,j \neq i $. 

    Moreover, if the distance between the fixed points sets in $\FF$ of any distinct pair of $\phi_i, \phi_j$ is at least $5$, then $[g]$ is weakly attracted to every attracting lamination of $\phi_j$ for every $j\neq i$. 
\end{corollary}
\begin{proof}
    For concreteness, suppose $\phi_1$ fixes the conjugacy class of a rank $1$ free factor $[\langle g \rangle]$. Let $\Lambda^+_j$ be any attracting lamination of $\phi_j$, $1\neq j$. If $[g]$ is carried by $[B]\in \mathcal{A}_{na}(\Lambda^+_j) = \mathcal{F}_j \cup [F_{m_j}]$, then either $[B]$ is a free factor component of $ \mathcal{F}_j$ or $[B] = [F_{m_j}]$. The former  case directly violates the sufficiently different hypothesis. 
    
    To complete the proof, it remains to show that if the distance between the fixed points sets in $\FF$ of any distinct pair of $\phi_i, \phi_j$ is at least $5$, then $[g]$ is weakly attracted to every attracting lamination of $\phi_j$ for every $j\neq i$. To see this, observe that if $[g]$ is carried by $[F_{m_j}]$, then the free factor conjugacy class $[\langle g \rangle]$ is disjoint from $\mathcal{F}_{supp}(\Lambda^+_j)$ and hence their distance in $\FF$ is at most $4$ - violating our hypothesis.  
\end{proof}

In Proposition \ref{nrh} we identified sufficient conditions for a collection of exponentially growing outer automorphisms to not have a common admissible subgroup system. In the following lemma, we have a hypothesis which negates that sufficient condition and examine the impact of our hypothesis on the nonattracting sinks for a collection of sufficiently different outer automorphisms. As it turns out, the nonattracting sinks in such cases will have a very rigid structure.

For an automorphism $\phi$, closed Nielsen paths are circuits which are in bijective  correspondence with  conjugacy classes which are fixed by  $\phi$. There are three primary cases in our analysis where Nielsen paths appear. First one (and the most important for us) as an indivisible Nielsen path associated to a geometric EG strata. The second one is from the ``twistor" in Dehn twist parts. The third case is when there is an isolated invariant rank-1 free factor, not covered by the previous two cases. All Nielsen paths are carried by the nonattracting sink. In case of absence of a free factor component  in the nonattracting sink, which is essentially what the hypothesis of Lemma \ref{rhconditions} ensures, the second and third cases of Nielsen paths are ruled out.

\begin{lemma}\label{rhconditions}
      Let $\phi_1, \cdots,  \phi_k$ be a collection of sufficiently different  exponentially growing outer automorphisms of $\out$. Let $\mathcal{K}^*_i$ be the nonattracting sink of $\phi_i$ for each $i$ and suppose that the following condition holds:
      \begin{itemize}
        \item  If $\mathcal{K}^*_i\neq \emptyset$ then for each $j\neq i$, there exists an attracting lamination $\Lambda^+_j$ of $\phi_j$ such that $\mathcal{K}^*_i$ is carried by $\mathcal{A}_{na}(\Lambda^+_j)$.
       \end{itemize}
      Then the following are true :
      \begin{enumerate}
          \item [a)] If $\mathcal{K}^*_i\neq\emptyset$ for some $i$, then $\mathcal{K}^*_i \neq\emptyset$ for every $i$. Moreover, each $\mathcal{K}^*_i$ has exactly one component. 
          \item [b)] If $\mathcal{K}^*_i\neq\emptyset$ for some $i$, then there exists a finite rank subgroup $F$ such that $\mathcal{K}^*_i = [F]$ for each $i$. 
          \item [c)] There exists an admissible subgroup system for the collection  $\phi_1, \cdots,  \phi_k$. 
         
      \end{enumerate}
\end{lemma}
\begin{proof}
    For concreteness assume that $\mathcal{K}^*_i = \mathcal{F}_i \cup [F_{i_1}] \ldots \cup [F_{i_n}]$ is nonempty and let $j\neq i$. Let $\Lambda^+_j$ be an attracting lamination of $\phi_j$ such that $\mathcal{A}_{na}(\Lambda^+_j) = \mathcal{F}_j \cup [F_j]$ carries $\mathcal{K}^*_i$. Sufficiently different condition implies that each component of the free factor system $\mathcal{F}_i$ must be  carried by $[F_j]$. Moreover, if the conjugacy class $[g_{i_k}]$ carried by the geometric component $[F_{i_k}]$, which represents the indivisible Nielsen path associated to the corresponding geometric strata, is a free factor conjugacy class, then $[F_{i_k}]$ must be carried by $[F_j]$ due to sufficiently different condition. On the other hand if $[g_{i_k}]$ is not a free factor conjugacy class and is carried by some component of $\mathcal{F}_j$, then the same component will also  carry the free factor support of $[g_{i_k}]$, hence the free factor support of the corresponding geometric lamination (which is invariant under $\phi_i$) - violating sufficiently different condition. Thus each $[F_{i_k}]$ (and hence entire $\mathcal{K}^*_i$) must be carried by $[F_j]$. This also proves that $\mathcal{K}^*_j$ is nonempty as $[F_j]$ is one of the components of the sink of $\phi_j$. 

    Now, switching the roles of $i, j$ in the above argument, we see that there exists an attracting lamination $\Lambda^+_i$ with a nonattracting subgroup system $\mathcal{F}'_i \cup [F_{i_l}]$ for some $1\leq l \leq n$ and entire $\mathcal{K}^*_j$ is carried by $[F_{i_l}]$. But we know that $[F_j]$ carries each $[F_{i_k}]$, thus proving that $[F_j] = [F_{i_l}]$ and $[F_{i_k}]$ is trivial if $k\neq l$. For the final step, observe that we have reduced to the case $\mathcal{K}^*_i = \mathcal{F}_i \cup [F_j]$ is carried by $\mathcal{A}_{na}(\Lambda^+_j) = \mathcal{F}_j \cup [F_j]$. If $[B]\in \mathcal{F}_i$ is nontrivial free factor component, as seen earlier - it must be carried by $[F_j]$, which violates malnormality of $\mathcal{K}^*_i$. Hence $\mathcal{K}^*_i = [F_j]$ has a single component. Since $i$ was chosen arbitrarily, this proves (a). 

    To prove (b), observe that in the paragraph above, we have actually shown that $\mathcal{K}^*_i = [F_j] = [F_{i_l}] = \mathcal{K}^*_j$. Since $i, j $ are arbitrary here, we may let $[F] = [F_j]$ and complete the proof of (b). 

    For the final part, we note that if all of the nonattracting sinks are empty, then we take $\mathcal{K} = \emptyset$ to be admissible subgroup system, otherwise we may take $\mathcal{K} = \{[F]\}$ as our admissible subgroup system. This completes  the proof.

\end{proof}


Having identified the structure of the nonattracting sinks for a sufficiently different collection of exponentially growing outer automorphisms, we proceed to answering the question about when we can expect extension groups to have a cusp-preserving relative hyperbolic structure. The answer breaks down into three cases. The first case that we deal with is the following proposition, when all the nonattracting sinks are nonempty. The remaining two cases when either all nonattracting are empty (Theorem \ref{suffdiffhyp}) or the mixed case with some being empty (Theorem \ref{RH}) are dealt with later.
\begin{proposition}\label{suffdiff}
  Let $\phi_1, \cdots, \phi_k$ be a collection of pairwise sufficiently different and exponentially growing outer automorphisms. If all the nonattracting sinks are nonempty, then exactly one of following conditions are true:
    \begin{enumerate}
        \item $ \phi_1, \cdots, \phi_k $ is a collection of fully irreducible geometric outer automorphisms induced by some pseudo-Anosov homemorphisms of the same compact surface with one boundary component. 
        Moreover, all sufficiently high powers of $\phi_i$' s generate a free group $Q$ such that $E_Q$ has a cusp-preserving relatively hyperbolic structure.
       \item $\mathcal{K}^*_i \neq \mathcal{K}^*_j$ for some $i\neq j$ and the extension $E_Q$ of a free group $Q$, which is generated by sufficiently high powers of $\phi_i$'s, cannot have a cusp preserving relatively hyperbolic structure. 
    \end{enumerate}
\end{proposition}
\begin{proof}
    When (1) holds, the work of Bowditch \cite{Bow-07} and Mj-Reeves \cite{MjR-08} for mapping classes of once punctured surfaces shows that sufficiently high powers of $\phi_i$'s generate a free group $Q$ of rank $k$ and the corresponding extension group $\F\rtimes \widehat{Q}$ is hyperbolic relative to the collection $\{\langle \sigma\rangle \oplus \langle \phi_i\rangle\}$, where $\sigma$ is the conjugacy class corresponding to the boundary of the surface. This gives us a cusp-preserving relatively hyperbolic structure and an admissible subgroup system $\{[\langle \sigma \rangle]\}$ for $Q$, implying that (2) fails.

    Suppose that (1) is false. We argue by contradiction to show that (2) must be true. Our argument involves showing that we cannot satisfy the hypothesis of Lemma \ref{rhconditions}, hence we must satisfy the hypothesis of Proposition \ref{nrh}. This will allow us to conclude that admissible subgroup system does not exist. Suppose on the contrary that we do indeed satisfy the hypothesis of Lemma \ref{rhconditions}. Conclusion $(2)$ of Lemma \ref{rhconditions} proves the existence of a finite rank subgroup $[F]$ such that $\mathcal{K}^*_i = [F]$ for each $i$. Moreover, the proof shows that this $[F]$ is actually a geometric component of $\mathcal{K}^*_i$, for each $i$, i.e., for each $\phi_i$ there exists some (geometric) attracting lamination $\Lambda^+_i$ such that $\mathcal{A}_{na}(\Lambda^+_i) = \mathcal{F}_i \cup [F]$ for some free factor system $\mathcal{F}_i$.
    
    Note that if $F$ is not infinite cyclic, then this means there are NEG edges attached to the boundary circle of the surface which supports the lamination $\Lambda^+_i$ for each $i$. Using the fact that $\mathcal{A}_{na}(\Lambda^+_i)$ cannot carry the free factor support of $\Lambda^+_i$, we may conclude that if $[F]$ is not cyclic, then there is a common free factor conjugacy class (namely, the free factor support of the surface noted earlier)  which is left invariant by each $\phi_i$ - thus  violating sufficiently different hypothesis. Thus $F$ must be infinite cyclic, say $\langle c \rangle$, where $[c]$ represents the conjugacy class of the word representing the boundary of the surface which supports all the laminations $\Lambda^+_i$'s. If the free factor support of this surface was proper, then we would violate sufficiently different hypothesis. Hence the surface and therefore $[c]$ must fill, \emph{i.e.} $\mathcal{F}_{supp}([c]) = \F$. This implies that the geometric stratum corresponding to the  $\Lambda^+_i$ is a top stratum for each $i$. Moreover, a geometric lamination being bottommost (\cite[Proposition 2.15]{HM-20}) implies that any of the lower strata cannot be EG, since the lamination fills. But since the sink of each $\phi_i$ is just $[\langle c \rangle]$, the lower filtration is just trivial (contractible) in each case. 
    Hence $\mathcal{A}_{na}(\Lambda^+_i) = [F] = [\langle c \rangle]$ for each $i$. 
    
    Using \cite[Lemma 5.3]{Gh-16}, this implies that each $\phi_i$ is a geometric fully-irreducible map which is induced by a pseudo-Anosov homeomorphism on the same compact surface (with one boundary component), bringing us back to case (1) - which we had assumed to fail - hence giving us a contradiction. In conclusion, if (1) fails, then we must satisfy the hypothesis of Proposition \ref{nrh}, and hence conclusion (2) follows using Theorem \ref{main1}.

\end{proof}

\section{small displacement sets, weak attraction theory and further obstructions to (relative) hyperbolicity}\label{sec:4}

To further our analysis and suggested by the previous section, we will consider a \emph{small displacement sets} in $\FF$ for exponentially growing outer automorphisms. We will then investigate interactions between attracting/repelling laminations, their corresponding non-attracting subgroup systems and the  small displacements sets.

\subsection{Small displacement sets of partial fully irreducibles}

\begin{definition} Let $F$ be some free-factor of $\F$ which is invariant under $\psi\in\out$. If the restriction of $\psi$ to $F$ is fully irreducible element of $\rm Out (F)$, then we say that $\psi$ is partial fully irreducible on  $F$.
\end{definition}
 Consider the isometric action of $\out$ on the free factor complex, $\mathcal F\mathcal F$.
For an automorphism $\phi$ and constant $C >  0$, we define the \emph{small displacement set} of $ \langle \phi \rangle$ on $\mathcal F\mathcal F$ to be,
\[
  \mathcal S_C= \{x \in \mathcal F\mathcal F: \exists k \neq 0\,\, \text{such that}\,\, d(x, \phi^k(x)) \leq C \} \]

Below we will prove that a small displacement set of a partial fully irreducible outer automorphism has a diameter bounded from above. For that we will need a few definitions regarding $\FF$.

A \emph{marking} of a graph $G$ is a homotopy equivalence $\calR_n \rightarrow G$ from a rose $\calR_n$. A metric on $G$ is a function that assigns a positive number (length) to each edge of $G$.  
 The (unprojectivized) \textit{outer space} is a space of  marked metric graphs which is introduced by Culler and Vogtmann in \cite{Vogt1} as an analog of \Teich space. We will denote by $CV$ the \textit{projectivized} outer space, in which the graphs will all have total volume $1$. For the details we refer the reader to \cite{Vogt1} and \cite{VogtSurvey}.

We will use the coarse projection $\pi: CV \rightarrow \FF$ defined as follows. For each proper subgraph $\Gamma_0$ of a marked graph $G$ that contains a circle, its image in $\FF$ is the conjugacy class of the smallest free factor containing $\Gamma_0$. Now by \cite{BF-12}, for two such proper subgraphs $\Gamma_1$ and $\Gamma_2$, $d_{\FF}(\pi(\Gamma_1),\pi(\Gamma_2) )\leq 4$ (Lemma 3.1, \cite{BF-12}). Then for $G\in CV$ we define,
\[ \pi(G):=\{\pi(\Gamma)|\,\Gamma \, \text{is a proper, connected, noncontractible subgraph of}\, G\}\]

We will call the induced map $CV\rightarrow \FF$ also $\pi$ which is clearly a \textit{coarse} projection in that the diameter of each $\pi(G)$ is bounded by $4$.

For a point $G\in CV$ and $A<\F$ we consider the core  graph $A|G$ corresponding to the conjugacy class of $A$. We will denote the projection of this core subgraph to $\FF(A)$ by $\pi_A(G)$.

In other words, $\pi_A(G)=\pi(A|G)$ where $\pi: CV \rightarrow \FF$ is the coarse Lipschitz map that takes a graph to the collection of all the proper subgraphs it contains. The image of this projection is of diameter at most $4$, by Bestvina-Feighn (\cite{BF-12}).
The pulled back metric via the immersion $p: A|G \rightarrow G$ gives that $A|G\in CV(A)$, the outer space of $A$.

Let $A$ and $B$ be conjugacy classes of two free factors. Assume that $A$ and $B$ are not disjoint, in other words they are not free factors of a common free splitting of $\F$ and one of them is not included in  the other. In this case $A$ and $B$ said to \emph{overlap}. For two overlapping conjugacy classes of free factors, we have the following definition.
\begin{definition}(\cite{BestSF, ST1,ST2}) Let $A,B$ be two overlapping free factors with rank of $A$ at least $2$. Then the \emph{subfactor projection}  $\pi_A(B)$ of $B$ to $A$ is defined to be,
\[ \bigcup \{\pi_A(G): G\in CV\,\, \text{and}\,\, B|G\subset G \}\]
where $B|G\subset G$ will mean embedding of the core graph  $B|G$ in  $G$.
\end{definition}
In  other words, given any tree $T$ with a vertex stabilizer $B$ and $\F$ action, $A$ fixes a tree $T^A$. If $T^A$ is not degenerate, which is guaranteed by overlapping condition,  the induced action of $A$ on $T^A$ gives a set of vertex stabilizers; by the Bass--Serre theory. $\pi_A(B)$ is defined to be this set of vertex stabilizers.\\
\textbf{Remark: } Any two rank $1$ free factors of $\F$ are either disjoint or equal. There is no possibility of overlap. Hence the definition of subfactor projections uses rank $\geq 2$.

\begin{lemma}[Bounded Geodesic Image Theorem \cite{ST1}, \cite{BestSF}] \label{BGIT}

For free group of rank $n\geq 3$ there is a number $ M_0\geq 0 $ such that if $A$ is a free factor and  $\gamma$ is a geodesic in $\mathcal F\mathcal F$ such that every vertex of $\gamma$ meets $A$. Then
\begin{center}
diam $\{\pi_A(\gamma) \}\leq M_0$
\end{center}
\end{lemma}

\begin{proposition} \label{diam} Let $\phi$ be partially fully irreducible automorphism with respect to a free factor $F^1$ with minimum translation length $\lambda$ on $F^1$. Then, whenever $\lambda \geq M_0$ and for every $C > 0$, there exists a sufficiently large constant $D=D(C,\lambda)$ such that the diameter of the small displacement  set
\[
\mathcal S_{\phi,C}= \{x \in \mathcal F\mathcal F: \exists k \neq 0\,\, \text{such that}\,\, d_{\FF}(x, \phi^k(x)) \leq C \} \]
corresponding to $ \langle \phi \rangle$ is bounded above by $D$.
\end{proposition}
\begin{proof}
Let $x\in S_C$. Then, by definition of this set, for some $k>0$ we have $d_{\mathcal F\mathcal F}(x, \phi^{k}(x))\leq C$. Let also $G\in CV$ such that $\pi(G)=x$.  Now, let $\{G_t\}$ be a folding path in $CV$ between $G$ and $\phi^k(G)$. The projection path $\pi(\{G_t\})$ in $\mathcal F \mathcal F$ is an unparametrized quasi geodesic between $x$ and $\phi^{k}(x)$ and it is Hausdorff close to a geodesic (\cite{BF-12}).\\
Now, since $F^1$ is of rank at least 2, the subfactor projection to the free factor complex $\FF(F^1)$ is coarsely defined. Moreover, by the hypothesis (or for a suitable power $>k$ of $\phi$),
$d_{F_1}(\tilde x, \phi^k_{|F^1}(\tilde{x}))>M_0$ where $\tilde{x} $ is the subfactor projection of $x$ to $F^1$ and $M_0$ is the constant from the Lemma \ref{BGIT}. By Lemma \ref{BGIT} again, there is a vertex $F^2$ along $\pi(\{G_t\})$ which does  not project to $\FF(F^1)$. By \cite{ST1} this means that either $F^1$ and $F^2$ are disjoint or one is included in the other as subgroups. Hence we have $d_{\mathcal F\mathcal F}(F^1, F^2)\leq 4$.\\
Use triangle inequality to deduce that $d_{\mathcal F\mathcal F}(x, F^1) \leq 4+C+C_1$ where $C_1$ is the Hausdorff distance between $\pi(\{G_t\})$ and the geodesic between $x$ and $\phi^{k}(x)$. Hence we have, $D=2(4+C+C_1)$.

\end{proof}
Following proposition captures the impact of the separation of small displacement sets of automorphisms $\phi_i, \phi_j$, $i\neq j$ on the weak attraction property of their corresponding lamination sets.  We continue using the same notations we set up at the beginning of this subsection (see the paragraph  before Proposition \ref{dist2wat}).

Thus, let $\mathcal \mathcal S_{\phi}$ denote the small displacement set of $\langle \phi\rangle $ in $\FF$.  We define the distance between the small displacement sets of $\phi_i$ and $\phi_j$ in $\FF$ to be,

\[d_{\FF}(\mathcal S_{\phi_i}, \mathcal S_{\phi_j}):=\min \{d_{\FF}([A], [B]):\,\,\, A\in \mathcal S_{\phi_i}, B \in \mathcal S_{\phi_j}\}   \]
\begin{proposition}\label{pfisep}
     Let $\phi_i$ be and partially fully irreducible on $F^i$ for all $i \in \{ 1,\cdots,  k\}$. Assume  $d_{\FF}(\mathcal S_{\phi_i}, \mathcal S_{\phi_j})\geq 5$ for all $i\neq j $. Then:
     \begin{enumerate}
         \item  There exists some $M$ such that whenever $m_i \geq M$, for all $i$, $\langle \phi_1^{m_1}, \cdots \phi_k^{m_k}\rangle$ is a free group.
         \item For $i\neq j$, $\phi_i$ and $ \phi_j$ do not have any common invariant free factors.
 
         \item  For all  $i \in \{ 1,\cdots,  k\}$, let $\Lambda^+_i$ be any attracting lamination of $\phi_i$ and  $\mathcal{A}_{na}(\Lambda^\pm_i)$ be nontrivial. For  a pair $\phi_i, \phi_j$, $i\neq j$, let  $\{[A^i_k]\}_{k = 1}^p = \mathcal{A}_{na}(\Lambda^\pm_i)$ and $\{[B^j_s]\}_{s = 1}^q = \mathcal{A}_{na}(\Lambda^\pm_j)$. Then, $\{[A^i_k \cap B^j_s]\}_{k, s}$ is a malnormal subgroup system. 
         \item No generic leaf of any attracting or repelling lamination of a $\phi_i$ is carried by a nonattracting subgroup system for  \emph{any} nongeometric dual lamination pair of $\phi_j$, for $j \neq i$.
        \item If $[g]$ is a nontrivial conjugacy class of a free factor which is fixed by $\phi_i$, then $[g]$ is cannot be carried by the nonattracting subgroup system of any attracting lamination of any $\phi_j$ for $j\neq i$.

    \end{enumerate}
\end{proposition}
\begin{proof}
    (1) is straightforward using Corollary \ref{pfifbf}. 
    
    Any invariant free factor of $\phi_i$ is contained in the respective small displacement set. When the respective small displacement sets are disjoint, the distance  between any invariant free factor of $\phi_i$ and any invariant free factor of $\phi_j$ is greater than $0$. Hence we cannot have a common invariant free factor. This proves (2).

    Every free factor component of $\mathcal{A}_{na}(\Lambda^\pm_i)$ is a vertex in the small displacement set of $\phi_i$. When the small displacement sets are of distance at least $5$, every free factor component of $\mathcal{A}_{na}(\Lambda^\pm_i)$ of rank $\geq 2$ meets every free factor component of $\mathcal{A}_{na}(\Lambda^\pm_j)$ of rank $\geq 2$. 
Malnormality of the collection of subgroups $\{[A^i_k \cap B^j_s]\}_{k, s}$ follows directly from the malnormality of $\mathcal{A}_{na}(\Lambda^\pm_i)$, for all $i$.  This proves (3).

 Let $\Lambda^+_j$ be any attracting lamination for $\phi_j$ with nontrivial nonattracting subgroup system $\mathcal{A}_{na}(\Lambda^\pm_j)$. 
 By definition every free factor component of $\mathcal{A}_{na}(\Lambda^\pm_j)$ is invariant under $\phi_j$. Hence every free factor component of  $\mathcal{A}_{na}(\Lambda^\pm_j)$ is an element of the small displacement set of $\phi_j$.  Let  $\ell$ be a generic leaf of some attracting lamination $\Lambda^+_i$, $i\neq j$ and let $[B]\in \mathcal{A}_{na}(\Lambda^\pm_j)$ be some free factor component of rank $\geq 2$. Assume $\ell$ is carried by $[B]$. Then $\Lambda^+_i$ is carried by $[B]$ which implies that $\mathcal{F}_{supp}(\Lambda^+_i) < [B]$, since $\mathcal{F}_{supp}(\Lambda^+_i)$ is the smallest free factor system that contains $\Lambda^+_i$. But $\mathcal{F}_{supp}(\Lambda^+_i)$ is invariant under $\phi_i$ and hence $\mathcal{F}_{supp}(\Lambda^+_i) \in \mathcal S_{\phi_i}$  which violates $d_{\FF}(\mathcal S_{\phi_i}, \mathcal S_{\phi_j})\geq 5$ when $i\neq j $. This completes the proof of (4). 
 
Let $[\langle g \rangle]$ be the conjugacy class of a rank $1$ free factor fixed by $\phi_1$. If $j\neq 1$ and $\Lambda^+_j$ is an attracting lamination of $\phi_j$, with $\mathcal{A}_{na}(\Lambda^+_j) = \mathcal{F}_j \cup [F_{m_j}]$. If $[g]$ is carried by $\mathcal{F}_j$, we would violate the distance between small displacement sets $\geq 5$ condition. So suppose that $[g]$ is carried by $[F_{m_j}]$. In this case, note that the free factor support of $\Lambda^+_j$ is not carried by $[F_{m_j}]$ and hence the free factor conjugacy classes $[\langle g \rangle]$ and $\mathcal{F}_{supp}(\Lambda^+_j)$ are disjoint and therefore violate that the small displacement sets have distance $\geq 5$. Therefore, $[g]$ cannot be carried by $\mathcal{A}_{na}(\Lambda^+_j)$. 
\end{proof}

Following proposition produces examples similar to Example \ref{notrh} when the small displacement sets are sufficiently separated.
\begin{proposition}  Let $\phi_i$ be and partially fully irreducible  on $F^i$  and assume $d_{\FF}(\mathcal S_{\phi_i}, \mathcal S_{\phi_j})\geq 5$ for all $j\neq i$. Suppose that some $\phi_i$ fixes a rank $1$ free factor $[g]$. Then  $\F \rtimes \widehat{Q}$ is not hyperbolic relative to any collection of subgroups so that cusps are preserved.
\end{proposition}
    \begin{proof}
       We note that by Proposition \ref{pfisep} (item (5)) - $[g]$ is is not carried by the nonattracting subgroup system of any attracting lamination of any $\phi_j$ for $j\neq i$. This  tells us that $[g]$ is weakly attracted to every attracting lamination of $\phi_j$, $j\neq i$. Using the fact that $[g]$ is carried by $\mathcal{K}^*_i$, Proposition \ref{nrh} now completes the proof.
    \end{proof}

\section{Geometry of Free by Free extensions}\label{sec:Hyp}

\subsection{Hyperbolicity of the extension from independent atoroidal automorphisms}
Now we are ready to characterize hyperbolicity of the extension group when automorphisms are atoroidal.
   \begin{hypnas}

    Let $\phi_1,  \ldots , \phi_k $ be a collection of exponentially growing outer automorphisms  such that no pair of automorphisms have a common power. 
  Then the following are equivalent:
    \begin{enumerate}
        \item  Each $\phi_i$ is atoroidal outer automorphism and $\mathcal L^{\pm}(\phi_i)\cap  \mathcal L^{\pm}(\phi_j)=\emptyset$ for all $i\neq j$.
        \item  There exists $M > 0$ such that for all $m_i \geq M$, $Q = \langle \phi_1^{m_1},\cdots,\phi_k^{m_k} \rangle$ is a free group and the extension group $\F\rtimes \widehat{Q}$ is hyperbolic (where $\widehat{Q}$ is any lift of $Q$).
    \end{enumerate}
    \end{hypnas}

\begin{proof} We prove the result for $k = 2$ for the sake of simplicity of the notation. General case is an exact replica of the argument.  

    $(1) \implies (2):$ Since $\phi_1, \phi_2$ are both hyperbolic,  by \cite[Lemma 3.1]{Gh-20} every conjugacy class grows exponentially under iteration by both of them. So $\mathcal{K}^*_1 = \mathcal{K}^*_2 = \emptyset$. We will use Theorem \ref{main1} for $\mathcal{K} =\mathcal{K}^*_1 = \mathcal{K}^*_2 = \emptyset$ (there is no restriction on $\mathcal{K}$ being nonempty in that theorem). To this end, all that remains is to show that the set of laminations are pairwise disjoint. Let $f: G\to G$ be a CT--map  representing $\phi_1$. If a generic leaf $\ell_1\in \Lambda^+_1\in \mathcal L^{\pm}(\phi_1)$ has a common end with a generic leaf $\ell_2\in \Lambda^+_2\in \mathcal L^{\pm}(\phi_2)$, then birecurrence implies that both ends of $\ell_i$ have height $s$, where $H_s$ is the exponentially growing strata corresponding to $\Lambda^+_1$ and $\ell_1, \ell_2 \subset G_s$. Since $\ell_2$ is asymptotic to $\ell_1$ and being weakly attracted to the dual lamination $\Lambda^-_1$ is an open condition, $\ell_2$ cannot be  weakly attracted to $\Lambda^-_1$ under iteration by $\phi_1^{-1}$. Using \cite[Proposition 6.0.8]{BFH-00}, we get that $\ell_2$ is a generic leaf of $\Lambda^+_1$. Hence $\Lambda^+_1 = \overline{\ell}_2 = \Lambda^+_2$, violating our hypothesis. Hence (2) follows from Theorem \ref{main1} as $\F\rtimes \widehat{Q}$ is hyperbolic relative to empty sets by this theorem.

    $(2) \implies (1):$ Hyperbolicity of $\F\rtimes \widehat{Q}$ implies that the Cannon-Thurston map for the inclusion $\iota \co \F \rightarrow \F\rtimes \widehat{Q}$ exists. Since $\phi_1, \phi_2$ do not have a common power, $\phi_1^\infty, \phi_2^\infty$ represent two distinct points in the boundary of $\F\rtimes \widehat{Q}$ corresponding the forward end of the axis generated by the elements $\phi_1, \phi_2$ respectively. Using \cite[Proposition 5.1]{Mj-97}, we see that the ending lamination sets of corresponding to $\phi_1^\infty$ and $\phi_2^\infty$ must be disjoint. Since \cite[Theorem 3.1, Lemma 4.4]{Gh-20} guarantees that generic leaves of all attracting laminations for $\phi_i$ is contained in the ending lamination set corresponding to $\phi^\infty_i$ and we know that the ending lamination set corresponding to any boundary point is a closed set, we get that $\phi_1$ and $\phi_2$ have no common attracting lamination. Similar arguments, while working with pairs of distinct boundary points $\{\phi_1^\infty, \phi_2^{-\infty}\}, \{\phi_1^{-\infty}, \phi_2^{-\infty}\}, \{\phi_1^{-\infty}, \phi_2^{\infty}\}$,  gives us $\mathcal L^{\pm}(\phi_1)\cap  \mathcal L^{\pm}(\phi_2)=\emptyset$.
\end{proof}

A simple application of Theorem \ref{main1} given below demonstrates a very easy way to construct free-by-free hyperbolic groups. Note that we neither need the automorphisms to be fully irreducible, nor do we need quasi-isometrically embedded orbits for the quotient group for the hypothesis.  We point out to the reader  that the converse to the corollary is obviously false. 
\begin{hypffc}
     Let $\phi_1,  \ldots , \phi_k $ be a collection of atoroidal  elements which do not have a common power. If no pair $\phi_i, \phi_j$, $i\neq j$ fixes a common vertex in the free factor complex $\FF$, then
     \begin{enumerate}
         \item There exists $M > 0$ such that for all $m_i \geq M$, the group $Q=\langle \phi_1^{m_1},\cdots,\phi_k^{m_k} \rangle$ is a free group, 
         \item the extension group $\F\rtimes \widehat{Q}$ is a hyperbolic group (where $\widehat{Q}$ is any lift of $Q$).
     \end{enumerate}
\end{hypffc}

\begin{proof}
    Let  $\Lambda^+\in\mathcal L^{\pm}(\phi_i)\cap  \mathcal L^{\pm}(\phi_j)$  be an attracting lamination for some $i,\neq j$. If $\mathcal{F}_{supp}(\Lambda)$ is proper, then $\phi_i, \phi_j$ both fix the vertex corresponding to $\mathcal{F}_{supp}(\Lambda)$, giving us a contradiction. If $\Lambda$ fills $\F$, then Proposition \ref{fillinglam} gives us a contradiction. Hence $\mathcal L^{\pm}(\phi_i)\cap  \mathcal L^{\pm}(\phi_j) = \emptyset$ and the conclusions follow from Theorem \ref{hypnas}.
\end{proof}

\subsection{Hyperbolicity and relative hyperbolicity from sufficiently different automorphisms}

In this section we will focus more on  the action on $\FF$  of our outer automorphisms; fixing some free factors.  

\begin{suffdiffhyp}
          Let $\phi_1, \ldots , \phi_k $ be a  collection of exponentially growing and sufficiently different  outer automorphisms.
      Let $Q = \langle \phi_1^{m_1},\cdots,\phi_k^{m_k} \rangle$ be a free group for sufficiently large $m_i$. 

 Then, $\F\rtimes \widehat{Q}$ is a hyperbolic group if and only if $\mathcal{K}^*_i = \emptyset$ for all $i\in \{1,\cdots, k \}$. 
\end{suffdiffhyp}

\begin{proof} 
    Suppose that $\mathcal{K}^*_i = \emptyset$ for all $i\in \{1,\cdots, k \}$. 
    If none of the  $\phi_i$ fixes a vertex in $\FF$, then they are all  fully irreducibles which do not have a common power, and we are done by \cite[Theorem 5.2]{BFH-97}. 
    Suppose some $\phi_i$ fixes a vertex in $\FF$ and  $\mathcal{K}^*_i=\emptyset$
for all $i\in \{1,\cdots, k \}$. To  apply Theorem \ref{hypnas}  it is  enough to show that $\mathcal L^{\pm}(\phi_i)\cap  \mathcal L^{\pm}(\phi_j)=\emptyset$ for all $j\neq i$.

   Since $\mathcal{K}^*_i$ is trivial every conjugacy class in $\F$ is weakly attracted to some attracting lamination of $\phi_i$. Hence $\phi_i$ is hyperbolic (atoroidal) outer automorphism for all $i\in \{1,\cdots, k \}$. 
   
    \textbf{Case 1:} Let $\Lambda^+_i$ be an attracting lamination of $\phi_i$ such that $\mathcal{F}_{supp}(\Lambda^+_i)$ is a proper free factor. Let $j\neq i$. If $\phi_j$ is fully irreducible, then $\Lambda^+_i$ cannot be an attracting or repelling lamination of $\phi_j$ since $\phi_j$ has unique attracting and repelling laminations which fill. If $\phi_j$ is reducible, then $\Lambda^+_i$ cannot be  an attracting or repelling lamination of $\phi_j$ for otherwise $\mathcal{F}_{supp}(\Lambda^+_i)$ would be a common invariant free factor and this contradicts the fact that any vertex fixed by $\phi_i$ is distance at least $2$ from any vertex fixed by $\phi_j$.

    \textbf{Case 2:} Suppose that $\Lambda^+_i$ fills. Pick some $j\neq i$. Then $\Lambda^+_i$ is not an attracting lamination of $\phi_j$ for otherwise hypotheses of Proposition \ref{fillinglam} are satisfied, which contradicts with the fact that $\phi_i, \phi_j$ do not have a common power. Same contradiction occurs if $\Lambda^+_i$ is a repelling lamination of $\phi_j$ by replacing $\phi_j$ with $\phi_j^{-1}$ in Proposition \ref{fillinglam}.
     Therefore $\phi_i$ and $\phi_j$ do  not have any common attracting or repelling laminations.
     Now apply Theorem \ref{hypnas} to get hyperbolicity of $\F\rtimes \widehat{Q}$.

    The converse part easily follows from the  observation that hyperbolicity of $\F\rtimes \widehat{Q}$ implies $\phi_i$, $\phi_j$ are both hyperbolic outer automorphisms for all $i\neq j$, $i,j \in \{1,\cdots, k \} $ and using \cite[Lemma 3.1]{Gh-20}. 

\end{proof}

The following theorem gives a necessary and sufficient condition for $\F\rtimes \widehat{Q}$ to have a cusp preserving relative hyperbolic structure when $Q$ is generated by a collection of sufficiently different outer automorphisms. One should read the following result as a relative hyperbolic  analogue of Theorem \ref{suffdiffhyp}.
\begin{theorem}\label{RH}
 Let $\{\phi_1, \cdots, \phi_k\}$ be a collection of  pairwise sufficiently different   and exponentially growing outer automorphisms. Let $\mathcal{K}^*_i$ be the nonattracting sink of $\phi_i,\,\,i \in \{1, \cdots, k\}$ and assume that $\mathcal{K}^*_j\neq \emptyset$ for some fixed $j$. Then the following are equivalent:
 \begin{enumerate}
     \item $\F\rtimes \widehat{Q}$ has a cusp preserving relatively hyperbolic structure where $Q$ is a free group generated by sufficiently large exponents of $\phi_i$'s.
     \item $\mathcal{K}^*_i = \mathcal{K}^*_j$ for all $i$. 
 \end{enumerate}
\end{theorem}
 
    \begin{proof}
     Lemma \ref{pfifbf} ensures high enough $m_i$'s so that $Q = \langle \phi_1^{m_1}, \ldots \phi_k^{m_k}\rangle$ is a free group. If $(2)$ holds, then (1) immediately follows from Theorem \ref{main1}. 
     
     Suppose next that for a free group $Q$ generated by sufficiently large exponents of $\phi_i$'s, $\F\rtimes \widehat{Q}$ has a relatively hyperbolic structure where cusps are preserved. By Proposition \ref{nrh} $\mathcal{K}^*_j$ must be carried by nonattracting subgroup system of some  $\Lambda_i\in\mathcal{L}^+(\phi_i)$ for each $i \in \{1, \cdots, k\}$, $i\neq j$. Conclusion (b) of Lemma \ref{rhconditions} completes the proof.  
 
\end{proof}

\subsection{\textbf{Relatively hyperbolic extensions using partially fully  irreducibles}}\label{sec:app2}

In this section we will give examples of relatively hyperbolic extensions using specific types of exponentially growing outer automorphisms. Our first example is for an exponentially growing outer automorphism which is \emph{partially} fully irreducible.

 Let $\phi_1, \cdots, \phi_k$ be a collection of partially fully irreducibles, each with respect to a free factor $F^i$ with $[F^i] \neq [F^j]$ whenever $i\neq j$. We denote the corresponding dual attracting and repelling lamination pair of each $\phi_i$ by $\Lambda^+_i$ and $\Lambda^-_i$. By partial fully irreducibility we have $\mathcal{F}_{supp}(\Lambda^\pm_i) = [F^i]$. We use these conventions below.

\begin{lemma}\label{pfiasymp}
    Let $\phi_i$ be  partially fully irreducible on $F^i$, $i \in \{1, \cdots , k\}$ Then $\Lambda^\pm_i$ and $\Lambda^\pm_j$ have asymptotic leaves if and only if $[F^i] = [F^j]$.
\end{lemma}
\begin{proof}
    Since the restriction of $\phi_i$ to $[F^i]$ is fully irreducible, every leaf of $\Lambda^+_i$ and of $\Lambda^-_i$ is generic. If $\ell_i\in \Lambda^+_i$ and $\ell_j\in \Lambda^+_j$ have a common end, birecurrency of generic leaves implies that $\ell_i$ is carried by $[F^j] = \mathcal{F}_{supp}(\Lambda^+_j)$. But since $ \mathcal{F}_{supp}(\Lambda^+_i) = \mathcal{F}_{supp}(\ell_i) = [F^i]$, we must have $[F^i] = [F^j]$. Converse direction is obvious.
\end{proof}

\begin{proposition}\label{dist2wat}
    Let $\phi_i$ be partially fully irreducible on  $F^i$, where $[F^i]\neq [F^j]$ whenever whenever $i\neq j$, $i,j \in \{1, \cdots , k\}$. Then, 
     \begin{enumerate}
         \item For all $ i \in \{1, \cdots , k\}$,  $[F^i]$ and free factor components of $\mathcal{A}_{na}(\Lambda^\pm_i)$ are close in the free factor complex $\FF$.
         \item For $i\neq j$, $\mathcal{A}_{na}(\Lambda^\pm_i)$ carries a leaf of $\Lambda^+_j$ if and only if $[F^j]$ is carried by $\mathcal{A}_{na}(\Lambda^\pm_i)$. Equivalently,  $\mathcal{A}_{na}(\Lambda^\pm_i)$ carries a leaf of $\Lambda^+_j$ if and only if $[F^j]$ is distance $\leq 1$ from some free factor component of $\mathcal{A}_{na}(\Lambda^\pm_i)$.
         \item If $[F^i]$ is at distance at least $2$ from every free factor component of $\mathcal{A}_{na}(\Lambda^\pm_j)$, then any generic leaf of $\Lambda^+_i$ is weakly attracted to $\Lambda^+_j$ under the action of $\phi_i$ $(\phi^{-1}_i)$, for all $i,j \in \{1, \cdots , k\},\,\, i\neq j$.
     \end{enumerate}

\end{proposition}

\begin{proof}
    Fix $i$. To prove (1), we show that any representative of $[F^i]$ and any representative of any free factor component of $\mathcal{A}_{na}(\Lambda^\pm_i)$ intersect at identity only. If they are not, there is some nontrivial conjugacy class $[g]$ of $F^i$  which is carried also by $\mathcal{A}_{na}(\Lambda^\pm_i)$. Since $\phi_i$ restricted to $F^i$ is fully-irreducible, this implies that $\Lambda^+_i$ must be geometric and  $[g]$ must be a representative of the unique closed indivisible Nielsen path corresponding to $\Lambda^+_i$. But this implies that $[g]$ cannot be a free factor of $\F$ and hence cannot be carried by any free factor component of $\mathcal{A}_{na}(\Lambda^\pm_i)$. The contradiction proves our claim. 
    
    Since the restriction of each $\phi_i$ to $F^i$ is fully irreducible, every leaf of $\Lambda^+_i$ is generic and hence not a circuit. If $\ell\in \Lambda^+_i$ is any leaf which is carried by $\mathcal{A}_{na}(\Lambda^\pm_j)$, $i\neq j$, malnormality of $\mathcal{A}_{na}(\Lambda^\pm_j)$ implies that there must be some free factor conjugacy class $[H]\in \mathcal{A}_{na}(\Lambda^\pm_j)$ which carries $\Lambda^+_i$ (since the set of lines carried by $[H]$ is closed and $\overline{\ell} = \Lambda^+_i$). By definition of free factor support we then get $F^i \leq H$ up to conjugation, which is equivalent to saying that the distance between $[F^i]$ and $[H]$ in the free factor complex of $\F$ is $\leq 1$. The converse follows directly from definitions. 
    
    The proof of (3) is exactly the same as the proof of \cite[Lemma 5.2]{Gh-23} (the cited proof does not anywhere use the assumptions of being atoroidal on the automorphisms) or one can see \cite[Remark 2.8, page 208]{HM-20}.
\end{proof}

\begin{lemma}
	\label{relhyp4}
  Let $\psi\in\out$  is partially fully irreducible with respect to $F$.
  Then there exists a $\psi$-invariant, malnormal subgroup system $\mathcal{K} = \{[K_1], \ldots , [K_p]\}$ such that the extension group $\F\rtimes \langle \psi \rangle$  is strongly hyperbolic relative
  to the collection of subgroups $\{K_s{\rtimes_\Psi}_s \mathbb{Z}\}_{s=1}^p$, where $\Psi_s$ is a
  chosen lift of $\phi$ such that $\Psi_s(K_s)=K_s$.
 \end{lemma}

  \begin{proof}
  If $\psi$ is partially fully irreducible with respect to $F$, then restriction of $\psi$ to $F$ is fully irreducible and hence there exists an attracting lamination $\Lambda^+$ supported by $[F]$.

  Let $\mathcal{K}$ be the nonattracting subgroup system associated to $\Lambda^+$. Now apply main result of \cite{Gh-23}  with $\mathcal{K}$ and $\psi$.
  \end{proof}

\begin{theorem}\label{pfirh}
       Let $\phi_i$ be   partially fully irreducible on  $F^i$, and $[F^i]\neq [F^j]$, $i,j \in \{1, \cdots , k\}$. Suppose that $[F^i]$ is distance $\geq 2$ from every free-factor component of $\mathcal{A}_{na}(\Lambda^\pm_j)$ for all $i = 1, \cdots ,k$ and $i\neq j$; if such free factor components exist. 
      Then there exists $M > 0$ such that for every $m_i\geq M$, $Q = \langle \phi_1^{m_1},\cdots,  \phi_k^{m_k}\rangle$ is a free group of rank $k$.

      Moreover, if $\mathcal{A}_{na}(\Lambda^\pm_1) = \mathcal{A}_{na}(\Lambda^\pm_2)= \cdots =\mathcal{A}_{na}(\Lambda^\pm_k)$, then for any lift $\widehat{Q}$ of $Q$,  $\F\rtimes \widehat{Q}$ is relatively hyperbolic group.
\end{theorem}

\begin{proof}

 Lemma \ref{pfiasymp} tells us that  $\Lambda^\pm_i$ and $\Lambda^\pm_j$ do not have any asymptotic leaves. This implies that we can choose long generic-leaf segments of $\Lambda^+_i$ and $\Lambda^-_i$ and construct attracting and repelling neighbourhoods $V^+_i$ and $V^-_i$ for each $i = 1, \cdots ,k$ so that $V^+_i, V^-_i, V^+_j, V^-_j$ are pairwise disjoint for $i\neq j$. 
 
    Since generic leaves of $\Lambda^+_i$ are not carried by $\mathcal{A}_{na}(\Lambda^\pm_j)$ for $i\neq j$, by Proposition \ref{dist2wat} generic leaves of $\Lambda^+_i$ are weakly attracted to $\Lambda^+_j$ (resp. to $\Lambda^-_j$) under the action of $\phi_j$ (resp. of $\phi_j^{-1}$). 
    
    The conclusion about $Q$ being free group now follows directly using \cite[Lemma 3.4.2, page 551]{BFH-00}.
  For the relative hyperbolicity of  $\F\rtimes \widehat{Q}$, let $\mathcal{K} = \mathcal{A}_{na}(\Lambda^\pm_1)$ and apply  of Theorem \ref{main1}.
\end{proof}
Note that the above theorem still remains true if we relax the requirement of $\mathcal{A}_{na}(\Lambda^\pm_1) = \mathcal{A}_{na}(\Lambda^\pm_2)= \cdots =\mathcal{A}_{na}(\Lambda^\pm_k)$ and replace it by the requirement of $Q$ having an admissible subgroup system.  

\subsection{Relatively hyperbolic extensions using relative fully irreducibles}

 Given a collection of free factors $F^1, F^2,...., F^p$ of $\F$ such that $\F= F^1\ast F^2\ast ... \ast F^p \ast B$ with $B$ possibly trivial,
we say that the collection forms a \emph{free factor system}, written as  $\mathcal{F}:=\{[F^1], [F^2],...., [F^p]\}$ and we say that
a conjugacy class $[c]$ of a word $c\in\F$ is carried by $\mathcal{F}$ if there exists some $1\leq s \leq p$ and a representative $H^s$ of $F^s$
such that $c\in H^s$.
\begin{definition}\label{def}
    We say that an automorphism $\phi\in\out$ is \emph{fully irreducible relative to } a free factor system $\mathcal{F}$ if $\mathcal{F}$ is $\phi$-invariant and there is no proper $\phi$-invariant free factor system  $\mathcal{F}'$ such that $\mathcal{F}\sqsubset \mathcal{F}'$.
\end{definition}

Relative fully irreducibles were developed by Handel and Mosher (\cite{HM-20})  to find a better analog of Ivanov's theorem on subgroups of mapping class groups of a surface so that the analogous theorem for subgroups of $\out$ works ``inductively'' on free factors that are fixed and that the behavior of the automorphism in between free factors is better understood. Below we will describe how  subgroups made of relative fully irreducibles give us easy ways to construct relatively hyperbolic extensions, under some mild restrictions on the generators. 

 Let $\mathcal{F} \sqsubset \mathcal{F}'$ be free factor systems, where $\mathcal{F}$ is a proper free factor system. If there exists a marked graph $G$ and realizations $H \subset H' \subset G$ of these free factor systems such that $H'\setminus H$ is a single edge, then $\mathcal{F}'$ is said to be an one-edge extension of $\mathcal{F}$. If no such realization exist, then $\mathcal{F}'$ is said to be a \emph{multi-edge extension} of $\mathcal{F}$.

We record the following lemma which follows from various nontrivial results in \cite{HM-20} and use the conclusion as our working definition of relative fully irreducible outer automorphism.

\begin{lemma}\label{firel}
 Suppose $\mathcal{F}\sqsubset \{[\F]\}$ is a multi-edge extension invariant under $\phi$ and every component of $\mathcal{F}$ is  $\phi-$invariant.
 If $\phi$ is fully irreducible rel $\mathcal{F}$ then there exists $\phi-$invariant dual lamination pair $\Lambda^\pm_\phi$ such that the following hold:

 \begin{enumerate}
 \item $\mathcal{F}_{supp}(\mathcal{F}, \Lambda^\pm) = [\F]$.
  \item If $\Lambda^\pm_\phi$ is nongeometric then  $\mathcal{A}_{na}(\Lambda^\pm_\phi)=\mathcal{F}$.
  \item If $\Lambda^\pm_\phi$ is geometric then there exists a root free $\sigma\in\F$ such that
  $\mathcal{A}_{na}(\Lambda^\pm_\phi) = \mathcal{F}\cup \{[\langle \sigma \rangle]\}$.
 \end{enumerate}

 Conversely, if there exists a $\phi-$invariant dual lamination pair such that if (1) and (2) hold  or if (1) and (3) hold then $\phi$ is fully irreducible rel $\mathcal{F}$.

\end{lemma}

\begin{proof}

Let $f: G\to G$ be an improved relative train track map representing $\phi$ and $G_{r-1}$ be the filtration element realizing $\mathcal{F}$.
 Apply \cite[Proposition 2.2]{HM-20}  to get all the conclusions for some iterate $\phi^k$ of $\phi$.

 We claim that $\Lambda^\pm_\phi$ obtained by applying \cite[Proposition 2.2]{HM-20} must be $\phi-$invariant.
 Otherwise, by  the definition of being fully irreducible relative to $\mathcal{F}$ 
 $\phi(\Lambda^+_\phi)$ will be an attracting lamination which is properly contained in $G_{r-1}$ and hence is carried by $\mathcal{F}$ which in turn
 is carried by the nonattracting subgroup system for $\Lambda^+_\phi$. This is a contradiction, hence $\Lambda^+_\phi$ is $\phi-$invariant. Similar arguments
 work for $\Lambda^-_\phi$.

 The converse part follows from  the case analysis in the proof of \cite[Theorem I, pages 18-19]{HM-20}.
\end{proof}



\begin{theorem}
    Let $\phi, \psi$ be fully-irreducible relative to multi-edge extension $\mathcal{F}$, with corresponding invariant lamination pairs $\Lambda^\pm_\phi, \Lambda^\pm_\psi$ (as in the equivalence condition \ref{firel}). If no leaf of  $\Lambda^+_\phi \cup \Lambda^-_\phi$ is asymptotic to any leaf of  $\Lambda^+_\psi \cup \Lambda^-_\psi$, then there exists an integer $M \geq 1$ such that

  \begin{enumerate}
  \item $Q = \langle \phi^m, \psi^n \rangle$ is free group of rank $2$ for all $m, n \geq M$.
  \item If $\Lambda^\pm_\phi, \Lambda^\pm_\psi$ are both nongeometric then the extension group $\F\rtimes \widehat{Q}$ is hyperbolic relative to the finite collection of subgroups $\{F_s\rtimes \widehat{Q}_s\}_{s=1} ^p$, where $\widehat{Q}_s$ is a lift that preserves $F_s$.
  \item If both $\Lambda^\pm_\phi$ and  $\Lambda^\pm_\psi$ are geometric laminations which come from the same surface, both fixing the conjugacy class $[\sigma]$ representing the surface boundary, then $\F\rtimes \widehat{Q}$ is hyperbolic relative to the finite collection of subgroups $\{\langle \sigma \rangle \rtimes \widehat{Q}_\sigma\}\cup\{F_s\rtimes \widehat{Q}_s\}_{s=1} ^p$, where $\widehat{Q}_s$ is a lift that preserves $F_s$ and $\widehat{Q}_\sigma$ is a lift that fixes $\sigma$.
  \end{enumerate}

\end{theorem}

\begin{proof}

To prove (1) we apply Theorem \ref{main1} with $\mathcal{K} = \mathcal{F}$ and  to prove (2) we apply Theorem \ref{main1} with $\mathcal{K} = \mathcal{F} \cup [\langle \sigma \rangle]$.

\end{proof}

As a corollary of the theorem for $\mathcal{F}=\emptyset$, we recover the case for surface group with punctures, which was proved in \cite[Theorem 4.9]{MjR-08}.

 \section{Further discussion: non-relative hyperbolicity }\label{sec:applications}

From the results we have so far we gather a necessary condition of non-relative hyperbolicity as follows:   
\begin{proposition}
Let  $\F\rtimes \widehat{Q}$ be NRH for every free group $Q$ which generated by nonzero powers of $\phi_i$'s which are sufficiently different.  
Then, there exists  $i \in \{1, \cdots, k\}$ such that $\mathcal{K}^*_i \neq \emptyset$ and $\mathcal{K}^*_i\neq \mathcal{K}^*_j$ for some $j \neq i$.
\end{proposition} 
 A sufficient condition for NRH remains to be unknown. However, we inversigate the following example. 
\begin{example}\label{notrh}

Let $\F = \langle a, b, c, d, e \rangle, F^1 = \langle a, b, c \rangle, F^2 = \langle a, b, c, d\rangle$. Define $\phi_1$ as the outer automorphism class of the map $\Phi_1\co a\mapsto ac, b\mapsto a, c\mapsto b, d\mapsto dc, e\mapsto ec$ and $\phi_2$ to be the outer automorphism class of the map $\Phi_2 \co a\mapsto ad, b\mapsto a, c\mapsto b, d\mapsto c, e\mapsto e$. Then $\mathcal{A}_{na}(\Lambda^\pm_1) = \emptyset$ and $\mathcal{A}_{na}(\Lambda^\pm_2) = \{[\langle e \rangle]\}$. In any relatively hyperbolic structure, the subgroup $\mathbb{Z} \oplus \mathbb{Z} \cong \langle e, \Phi_2 \rangle$ must be contained in some peripheral subgroup.  A simple computation shows that $\F \leq  \langle \Phi_1^{-1} \Phi_2 , b, c\rangle $ and hence the extension group $1 \to \F \to E \to \langle \phi_1, \phi_2 \rangle \to 1$ is not relatively hyperbolic ($NRH$ as in \cite{BDM-09}). 

The hypothesis for freeness in the above theorem is easily checked, and so there exists $M$ such that for every $m, n \geq M$, the group  ${Q} = \langle \phi_1^m, \phi_2^n\rangle$ is a free group of rank 2. Once we raise the automorphisms to sufficiently high powers, the abnormality with peripheral subgroups goes away. We ask the question: is $\F\rtimes \widehat{Q}$ hyperbolic relative to $\langle e, \Phi_2 \rangle$ when $Q$ is generated by high enough powers of  $\phi_1, \phi_2$ ?
\end{example}

 \begin{remark}\label{question}
 Note that $[F^i]$ is a fixed vertex for the action of $\phi_i$ on the free-factor complex  $\FF$, for $i = 1, 2$. Also the free-factors in the set $\mathcal{A}_{na}(\Lambda^\pm_i)$ are fixed vertices for the same action. If we translate the hypothesis of Theorem \ref{pfirh} in the framework of the free factor complex $\FF$ of $\F$, then what we require is that (a) $[F^1]$ and $[F^2]$ are distinct points in $\FF$ (b) $\mathcal{A}_{na}(\Lambda^\pm_1) = \mathcal{A}_{na}(\Lambda^\pm_2)$. Using Proposition \ref{dist2wat}, the second hypothesis in turn implies that $[F^1]$ and the free-factor components of $\mathcal{A}_{na}(\Lambda^\pm_2)$ are close in $\FF$. Similarly free-factor components of $\mathcal{A}_{na}(\Lambda^\pm_1)$ and $[F^2]$ are also close in $\FF$. By partial fully irreducibility  $[F^i]$ and free-factor components of $\mathcal{A}_{na}(\Lambda^\pm_i)$ are also close. Therefore the entire setup of Theorem \ref{pfirh} is in a small (bounded) region of $\FF$.
    The obvious question that comes to mind is: if we assume that $[F^1]$ and $[F^2]$ are far apart in Theorem \ref{pfirh} can we still expect relative hyperbolicity of $\F\rtimes \widehat{Q}$ ?
\end{remark}

\bibliographystyle{alpha}
\def\bibfont{\footnotesize}
\bibliography{biblo1}

\end{document}